\pdfoutput=1

\documentclass{amsart}

\usepackage{graphicx}
\usepackage{subcaption}
\usepackage{hyperref}
\usepackage[section]{placeins}

\newcommand{\Vol}{\mathit{Vol}_n}

\newcommand{\expect}[1]{\mathbb{E} \big[ #1 \big]}
\newcommand{\condexpect}[2]{\mathbb{E} \big[ \, #1 \,  \big| \, #2 \, \big]}

\newtheorem{thm}{Theorem}[section]
\newtheorem{cor}[thm]{Corollary}
\newtheorem{prop}[thm]{Proposition}
\newtheorem{lem}[thm]{Lemma}
\newtheorem{conj}[thm]{Conjecture}

\makeatletter
\let\c@equation\c@thm
\makeatother
\numberwithin{equation}{section}

\usepackage[foot]{amsaddr}

\title{Volumes of Random Alternating Link Diagrams}

\author{Malik Obeidin}
\address{University of Illinois at Urbana-Champaign}
\email{mobeidin@illinois.edu}
\date{\today}

\begin{document}
%--------Meta Data: Fill in your info------

\maketitle

\begin{abstract}
We describe a model of random links based on random 4-valent maps, which can be sampled due to the work of Schaeffer. We will look at the relationship between the combinatorial information in the diagram and the hyperbolic volume.  Specifically, we show that for random alternating diagrams, the expected hyperbolic volume is asymptotically linear in the number of crossings. For nonalternating diagrams, we compute the probability of finding a given, arbitrary tangle around a given crossing, and show that a random link diagram will be highly composite. Additionally, we present some results of computer experiments obtained from implementing the model in the program \emph{SnapPy}.

\end{abstract}

%\tableofcontents

\section{Introduction}

Random knots and links have been studied through a variety of avenues --- initially inspired by physical problems such as the knotting that occurs in bacterial DNA. Additionally, link exteriors form an important, classical family of 3-manifolds.  Since the set of links is countably infinite, one way to select randomly from this set is to filter links by some kind of complexity, such that the number of links of any given complexity is finite. From there, we can sample uniformly among links of a given complexity, and see what happens as the complexity increases without bound --- the choice of this complexity gives different models of random links, which can have different asymptotic behaviors.

%\iffalse %%%%%%%%%%%%%%%%%%%%%%%%%%

Some previously studied models of random knots and links include the Petaluma model \cite{petaluma}, random polygonal walks \cite{diaopippengersumners}, random braids \cite{ma}, and the Chebyshev billiard table model \cite{cohenkrishnan}. In this paper, we will examine a model which samples uniformly from (rooted) link diagrams of a given number of crossings. This model is also studied by Chapman in \cite{chapman}. In the alternating case, using the results of Lackenby, Thurston, and Agol in \cite{lackenby}, we will show that the expected volume grows linearly in the number of crossings of the link diagram.

\begin{thm}
\label{thm:expectedvolume}
Let $\Vol$ be the random variable which returns the hyperbolic volume of a random alternating link diagram with $n$ crossings. The expected hyperbolic volume is bounded by
\begin{equation*}
  \left( \frac{19v_3}{54} \right) n + \left( \frac{10v_3}{27}-2 \right) +O\left( \frac{1}{n} \right) \leq \expect{\Vol}  \leq  v_8 n
\end{equation*}
Numerically, the coefficients on $n$ of the lower bound and upper bounds are approximately 0.3571 and 3.6638, respectively. The constants $v_3$ and $v_8$ are the volumes of the hyperbolic ideal regular tetrahedron and octahedron, respectively.
\end{thm}

So, up to a term that goes to zero, the expected volume is bounded by (increasing) linear functions of the number of crossings.  In big theta notation, $\expect{\Vol} = \Theta(n)$. The links we get in this model are generically hyperbolic --- we also take the convention that the hyperbolic volume of a nonhyperbolic link is zero.

The bounds given in \cite{lackenby} are given in terms of the twist number of a diagram, not the crossing number. The main difficulty here is then the computation of the expected twist number in our model, which is in turn deduced from the expected number of bigons in the complement of the diagram. We only use the lower bound from \cite{lackenby} --- it was pointed out to the author by Stavros Garoufalidis that there is an asymptotically sharp upper bound for the hyperbolic volume due to a construction of Dylan Thurston which divides the complement of an $n$-crossing link diagram in $S^3$ into $n$ octahedra. Since the volume of such a hyperbolic octahedron is bounded by the volume of the ideal regular octahedron, we get an immediate linear upper bound of the volume as a function of the crossing number, and the bound on the expectation trivially follows \cite{garoufalidisle}.

In the nonalternating case, we show that we generically do not get hyperbolic links.  This occurs in other models as well due to the phenomenon of \emph{local knotting}: in ``small'' regions of the diagram, we have positive probability of getting any possible picture, including ones which obstruct hyperbolicity by forcing the link diagram to be a satellite link or unknotted/unlinked (see Figure~\ref{fig:forbiddentangle}). This occurs in other models as well --- for example, for Gaussian random polygons \cite{jungreis}. Our result gives a formula for the probability of a local picture occurring, and shows it depends only on the ``size'' of the picture. To be more precise, a local picture is a rooted tangle diagram with $n$ crossings and $2p$ points intersecting the boundary, which embeds around a given crossing.  We compute the formula for the probability here, and use it to show that one expects a given local picture to occur linearly often (with respect to the number of crossings):

\begin{thm}
\label{thm:genericallynothyperbolic}
Let $T$ be a rooted tangle with $n$ crossings and $2p$ boundary points. Then, the number of rootings $N_{T,c}$ of a random rooted $c$-crossing link diagram $L$ for which $T$ embeds around the root has expectation which is asymptotically linear in $c$, the number of crossings:
\[  \expect{N_{T,c}} = (4c) 2^{-n} P(n,p,c) + O(cb^{-c}) \]
for some $b \in (0,1)$. Normalizing by the number of crossings $c$, we have a positive limiting expectation:
\[ \lim_{c \to \infty} \frac{1}{c} \expect{N_{T,c}} =  \left( 4\cdot2^{-n} \right) \frac{(3p)!}{9\cdot p! (2p-1)!} \left( \frac{2}{3} \right)^{p-2} 12^{-n} > 0 \]
\end{thm}
See Section~\ref{nonalternatingdiagrams} and Proposition~\ref{prop:nonalternatingprob} for the definition and asymptotics of $P(n,p,c)$. Having this asymptotic behavior immediately yields the following information about a large random link, by applying Theorem~\ref{thm:genericallynothyperbolic} to various tangles.
\begin{cor}
\label{cor:genericallynothyperbolic}
For a random link diagram, the following quantities have expectation which is asymptotically bounded above and below by increasing linear functions of the number of crossings:
\begin{enumerate}
\item The number of link components
\item The number of pieces in the connect sum decomposition
\item The number of pieces of the JSJ decomposition of the exterior
\item The Gromov norm of the exterior
\item The crossing number
\end{enumerate}
\end{cor}

Notice that this means generically a random link diagram in our model is highly composite, hence nonhyperbolic. A similar result is shown in \cite{chapman}; the analogous theorem there states loosely that diagrams which do not contain a given local picture linearly often are exponentially rare among all link diagrams. He also shows this in the case of random knot diagrams, whereas we only consider link diagrams, which are simpler to work worth. However, here we can compute explicitly the probability of seeing a given tangle occuring.

We have also included the result of some computer experiments with the \emph{Spherogram} module in \emph{SnapPy} \cite{snappy} implementing this model.  In particular, we find that the hyperbolic volume appears to be quite strongly linear with respect to the crossing number for alternating links. For a given crossing number, we examine the (normalized) distribution of volumes, and present some evidence that it converges to a limiting distribution. However, this distribution is not normal. Finally, to see if the number of bigons and larger faces affects the volume on average, we generated large amounts of data relating these quantities, and find that in some sense, diagrams with larger faces tend to have less volume. The result in \cite{lackenby} suggests that as the twist number decreases, or alternatively, the number of bigons increases, the volume should decrease. What this data suggests that this remains true if we fix the number of bigons and vary the number of triangles, and so on.

\subsection*{Acknowledgments}  
The author was partially supported by NSF grant DMS-1510204 and Campus Research Board grant RB15127. The author is especially indebted to his advisor, Nathan Dunfield, for suggesting the problem, and for his wonderful help and insight.
The author would also like to thank Stavros Garoufalidis for pointing out the improved upper bound, and Dylan Thurston for his helpful comments.

\section{Outline of Paper}

In Section \ref{amodelforrandomlinks}, we describe the model we are using and set up our notation. In Section \ref{alternatinglinkdiagrams}, we show that the expected hyperbolic volume of a random alternating link in this model is linear in the number of crossings. We also present the results of numerical experiments about the behavior of the hyperbolic volume. In Section \ref{nonalternatingdiagrams}, we show that a random nonalternating diagram is typically composite, and compute the probability of seeing a given tangle as a local picture in a random diagram. Finally, in Section \ref{largerfaces}, we present the preliminary data showing the tendency of diagrams with larger faces to have higher hyperbolic volume, all else being equal.

\section{A Model for Random Links} \label{amodelforrandomlinks}

One natural way to randomly sample links is through link diagrams. A \emph{rooted planar map} is an equivalence class of embeddings of a planar graph into the plane, where the equivalence is given by orientation-preserving homeomorphisms of the sphere, and where one oriented edge called the \emph{root} is specified. The homeomorphism here must take root to root. Loops and multiple edges are allowed. Note that a \emph{4-valent} planar map (all vertices having valence 4) is a projection of a link in $\mathbb{R}^3$ into the plane; to specify a link from a 4-valent planar map, we think of the vertices as crossings and specify which strand goes over which. We will consider the set of all rooted 4-valent planar maps with $n$ vertices, which we denote $Q(n)$. The rooting here makes the problem of enumeration much simpler, as a rooted map cannot have any automorphisms preserving the root. We will also refer sometimes to the \emph{root vertex}, the source vertex of the root, and the \emph{root face}, the face on the right side of the root, when pointing in the direction of the root. 

 To randomly sample link diagrams of size $n$, we first sample uniformly at random from $Q(n)$, and then flip a coin at each vertex to determine which strand crosses over or under.  We can work with this model computationally due to the work of Schaeffer, who describes in \cite{schaeffer97} an algorithm to sample uniformly from $Q(n)$ in linear time, and whose software \emph{planarmap} \cite{planarmap} implementing this algorithm is used by \emph{SnapPy}.

If we wish to restrict to alternating links, then from a given planar map in $Q(n)$, there are two choices of alternating link diagrams whose ``shadows'' in the plane are that map. However, since the diagrams are rooted, we can fix a choice by stipulating that the root edge goes over-to-under.  Instead of working with $Q(n)$, which has diagrams which represent non-prime links, we will consider instead the sets
\[SQ(n)  =  \{\mbox{3-edge-connected, 4-valent, rooted, planar maps}\} \] 
and
\[AD(n) =  \{\mbox{prime, reduced, alternating, rooted link diagrams}\} \]

By a \emph{prime} diagram, we mean that the link diagram is not an ``obvious'' connect sum of two other diagrams, and by \emph{reduced}, we mean a diagram that has no loops --- no crossings that could immediately be removed with a type I Reidemeister move.  These sets are easily seen to be in bijection by our convention for choosing over/under data above; from now on, we will work exclusively with $SQ(n)$, and not explicitly mention the bijection. In \cite{menasco84}, it is shown that the alternating link diagram corresponding to a 3-edge-connected, 4-valent, rooted planar map is the diagram of a prime, alternating, non-split, link.  Additionally, such a link is either a torus link or hyperbolic.

\section{Alternating Link Diagrams} \label{alternatinglinkdiagrams}

\subsection{Enumeration and Sampling} \label{enumerationandsampling}
Various classes of planar maps were enumerated in a series of papers by Tutte and Brown \cite{tuttecensus, tutteenum}.  The class we are interested in was enumerated by Brown in \cite{brown62}, though in slightly disguised form.  The maps enumerated are \emph{nonseparable rooted planar maps with n edges} --- planar maps without a loop or \emph{cut-vertex}. A cut-vertex is a vertex $V$ which partitions the edges of the map into two sets which only share $V$ as a vertex.  This set is in bijection with $SQ(n)$ through the medial bijection (see Figure \ref{fig:medial}).

\begin{figure}[h]
\centering

\begin{subfigure}{.5\textwidth}
\centering
\includegraphics[width=.9\linewidth]{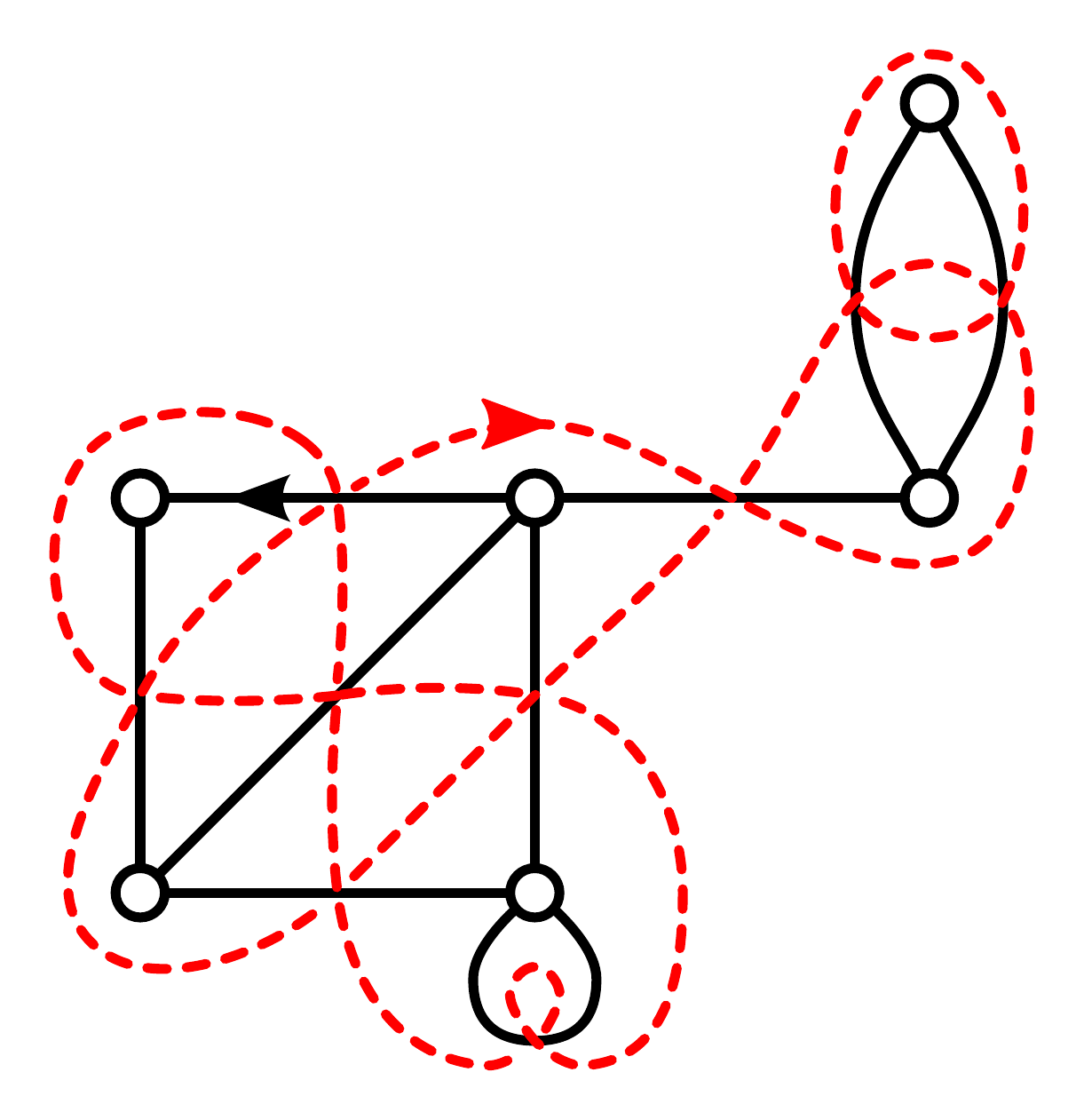}
\end{subfigure}%
\begin{subfigure}{.5\textwidth}
\centering
\includegraphics[width=.9\linewidth]{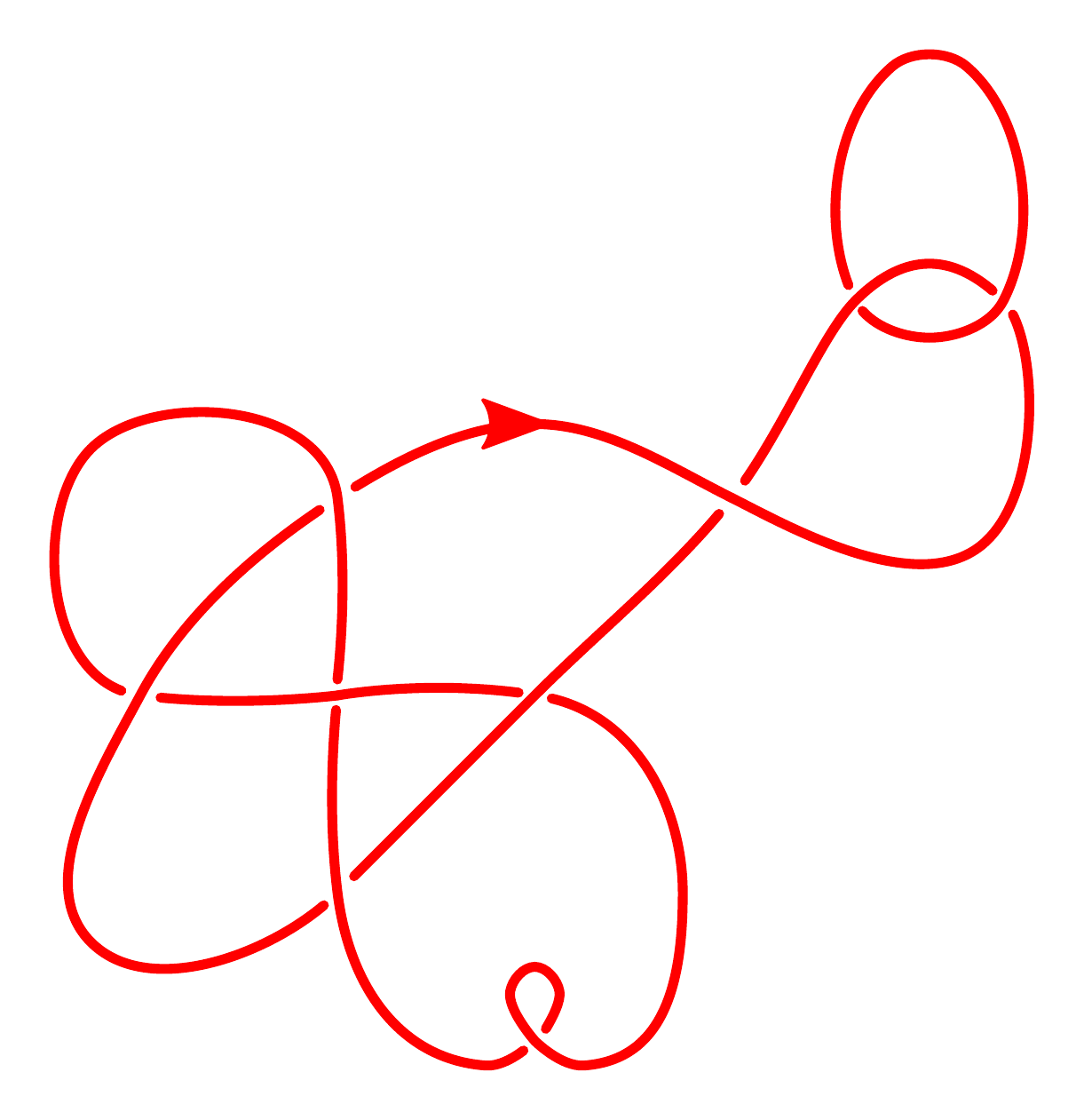}
\end{subfigure}
\caption{The medial bijection: to every planar map we can associate a 4-valent planar map by creating a vertex at the midpoint of each edge and connecting around each face.  The root of the new map is taken by convention to be the outward edge 2nd from the original clockwise, as shown. The example here is \emph{separable}: for example, the root vertex is a cut vertex. On the right, we have the resulting alternating link diagram.}
\label{fig:medial}
\end{figure}

Note that a cut vertex corresponds exactly to a place where the corresponding link diagram (under the medial bijection) can be separated into two halves, connected by two strands. The convention to transfer the root to the 4-valent map is as in Figure \ref{fig:medial}. One should be careful above the inverse of the medial bijection --- it appears that, around a vertex of a 4-valent map, there are four choices of oriented edge coming out, hence there should be four associated planar maps after applying the inverse. There appear to only be two, since we have two orientations of the edge that corresponds to that vertex in the planar map. However, there are two different (unrooted) planar maps that give the same (unrooted) 4-valent map after the medial bijection, a map and its dual, and each of these will have two rootings of a given edge. These four possibilities give the four rootings of the corresponding vertex after the medial bijection.

Jacquard and Schaeffer in \cite{jacquardschaeffer} describe an algorithm to sample efficiently from $SQ(n)$, so we can sample from prime alternating link diagrams of size $n$.  In our case, we wish to see that the volume indeed grows linearly in the size of the diagram, as one might expect from the following result from \cite{lackenby}:

\begin{thm}[Lackenby, Agol, Thurston 2004]
Let $D$ be a prime, alternating diagram of a hyperbolic link $K$.  Then
\begin{equation}
\label{lackenbyineq}
v_3(t(D)-2)/2 \leq \mbox{Vol}(S^3-K) < 10v_3(t(D)-1) 
\end{equation}
where $v_3 \approx 1.01494$ is the volume of a regular hyperbolic ideal tetrahedron.

Moreover, the upper bound is asymptotically sharp.
\end{thm}

Here $t(D)$ represents the \emph{twist number} of a diagram, which is defined in \cite{lackenby} as the number of twist regions. A \emph{twist region} is a maximal chains of bigons, arranged end to end, or a crossing not adjacent to any bigon. An example is given in Figure \ref{fig:twistnumber}.

\begin{figure}
\centering
\includegraphics[scale=.4]{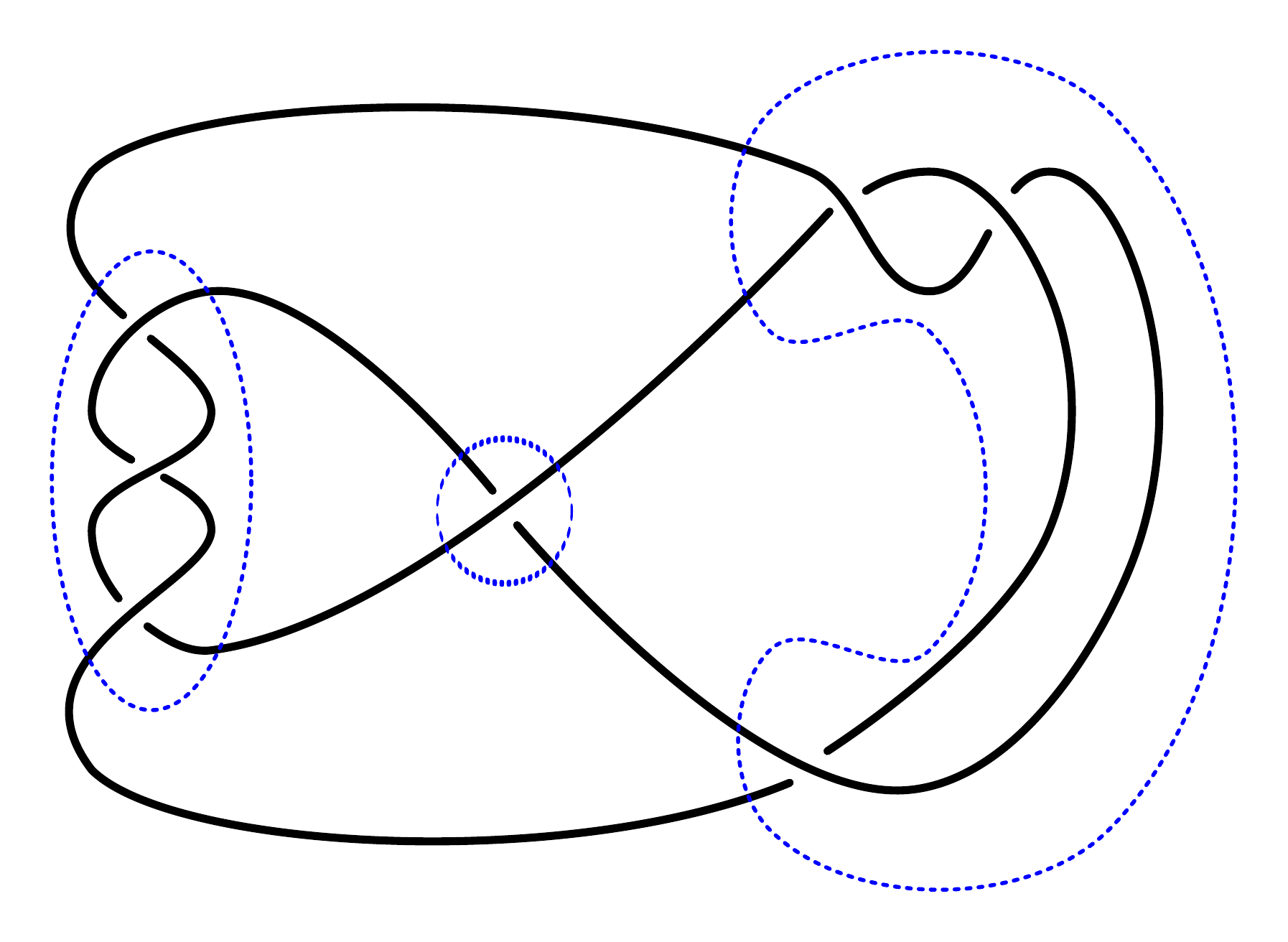}
\caption{This diagram has the 3 twist regions outlined.  Note that there are 7 crossings, and 4 bigons in the complement.}
\label{fig:twistnumber}
\end{figure}

Note that for each bigon chain, the number of crossings is one more than the number of bigons.  Since the twist regions partition the crossings and the bigons of $D$,
\begin{equation}
 \begin{split}
\label{twistsandbigons}
t(D) & = \sum_{\mbox{twist regions $T$}} 1 \\
     & = \sum_{\mbox{twist regions $T$}} [(\mbox{\# crossings in $T$}) - (\mbox{\# bigons in $T$})] \\
     & = (\mbox{\# crossings in $D$}) - (\mbox{\# bigons in $D$})
 \end{split}
\end{equation}

That is, the twist number of a diagram is simply the number of crossings minus the number of bigons in the complement. There is one corner case here where the formula doesn't hold: if the bigon chain connects to itself to make a loop, we then have a connected component of the link diagram which is the standard diagram of the torus link $T(2,n)$.  Since our diagram is assumed connected, there is only one such unrooted diagram to consider, which has two rootings in $SQ(n)$, as shown in Figure \ref{fig:toruslink}.  In this case, the twist number $t(D)$ is just 1.

\begin{figure}[b]
\centering
\includegraphics[scale=.8]{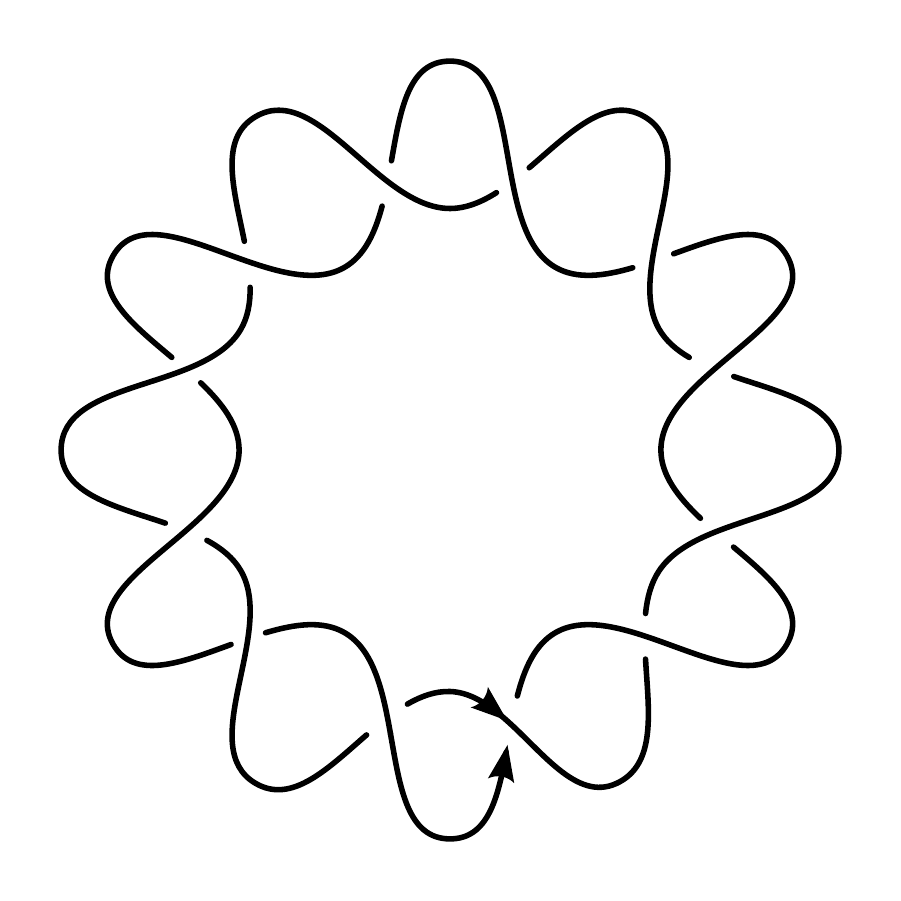}
\caption{The torus link $T(2, 12)$. The root edge will either have a bigon to its right, or a larger face. Those two cases, after homeomorphism of the sphere, can be transformed to the two rootings shown above.}
\label{fig:toruslink}
\end{figure}

So, in order to get bounds on the expected volume in our model, we can simply find the expected number of bigons.  It is Brown's enumeration that allows us to do this. We first find the probability that a randomly chosen element of $SQ(n)$ will have a bigon as its root \emph{face}.

%\[\mbox{image of root face?} \]

In the medial bijection between nonseparable planar maps and $SQ(n)$, a nonseparable planar map where the root vertex has valence $m$ corresponds to an element of $SQ(n)$ where the root face has $m$ sides.  Let $SQ(n,m)$ be the subset of all elements of $SQ(n)$ where the root face has $m$ sides. Then, we have the enumeration
\begin{thm}[Brown 1962]
\label{brownenumeration}
The number of nonseparable planar maps with $n$ edges and root valence $m$ is given by:
\[ |SQ(n,m)| = \left( \frac{m}{(2n-m)!} \right) \sum_{j=m}^{\min(n,2m)} \frac{(3m-2j-1)(2j-m)(j-2)!(3n-j-m-1)!}{(n-j)!(j-m)!(j-m+1)!(2m-j)!} \]
where $n \geq m \geq 2$.

The number of nonseparable planar maps with $n$ edges (and no other restriction) is 
\[ |SQ(n)| = \frac{2(3n-3)!}{n!(2n-1)!}\]
\end{thm}

Since we select uniformly from all diagrams, the probability that the root face has size $m$ is then given by 
\[P(n,m) = \frac{|SQ(n,m)|}{|SQ(n)|} \]

We will use this to compute the expected number of $m$-gons in a diagram chosen in our model.  Immediately, this gives us the following fact.

\begin{lem}
\label{fmexpectation}
Let $F_m$ be the random variable defined on $SQ(n)$ with the uniform probability given by
\[ F_m(D) = \begin{cases}
    1/m & \mbox{if root face is size $m$} \\
    0 & \mbox{otherwise}
   \end{cases}
 \] 
Then,
\[ \expect{F_m} = \frac{1}{m} \frac{|SQ(n,m)|}{|SQ(n)|} = \frac{1}{m} P(n,m)  \]
\end{lem}

In order to relate this to the number of $m$-gons in a diagram, we partition $SQ(n)$ into the equivalence classes of \emph{unrooted} diagrams.  That is, we group together all rooted diagrams which are different rootings of the same unrooted diagram. The useful and generic case here is when the underlying unrooted diagram has no nontrivial automorphisms --- orientation preserving homeomorphisms of the sphere taking the diagram to itself which actually permute the edges. We call such diagrams \emph{asymmetric}. Equivalently, a diagram is \emph{asymmetric} if there are exactly $4n$ inequivalent rootings of the unrooted diagram, one for each oriented edge. We denote the equivalence class of $D$ under this equivalence relation by $[D]$, and the number of $m$-gons in $D$ by $N_m(D)$.

To relate $F_m$ to our desired $N_m$, we have the following simple lemma.
\begin{lem}
\label{fmandnm}
Let $D \in SQ(n)$ be asymmetric, and $N_m$ and $F_m$ as above. Then, the conditional expectation over the equivalence class $[D]$ is
\[ \condexpect{F_m}{ [D]}  = \frac{1}{4n} N_m(D) \]
\end{lem}
\begin{proof}
If $D$ is asymmetric, then we are computing the average over all rootings $r$ of the underlying unrooted diagram, of which there are $4n$.  Calling $D_r$ the diagram in $[D]$ with root $r$, then we have
\[ \condexpect{F_m}{[D]} = \frac{1}{4n} \sum_r{F_m(D_r)} \]
For each root $r$, if the root face (the face to the right of $r$) is not an $m$-gon, then we get zero.  Otherwise, we get a contribution of $1/m$, which will occur for all $m$ edges around that face.  So, the sum gives $N_m(D)$, as desired.
\end{proof}

In order to now get estimates on the expectation of $N_m$, we need to know that ``most'' diagrams are asymmetric. Fortunately, this is a well studied problem; in particular, it is known that the proportion of nonseparable planar maps which are symmetric is exponentially small \cite{richmondwormald}. With these facts, we can complete the computation of the expectation of $N_m$.

\begin{thm}
\label{expectedbigons}
Let $N_m$ be the random variable on $SQ(n)$ with the uniform probability measure defined as above.  Then, the expectation of $N_m$ is given by
\[ \expect{N_m} = \frac{4n}{m} P(n,m) + O(na^n) \]
for some $a \in (0,1)$.
\end{thm}
\begin{proof}
Take representatives $A_i$, $i \in \{1, ..., k\}$, of each of the equivalence classes of asymmetric diagrams.  Then, by the law of total expectation, we can split the expectation of $N_m$ as follows:
\[\expect{N_m} = \sum_{i=1}^k\condexpect{N_m}{[A_i]}P([A_i]) + \condexpect{N_m}{S_n} P(S_n) \]
where $S_n$ is the set of all symmetric diagrams in $SQ(n)$.  The conditional expectation over the equivalence class $[A_i]$ of $N_m$ is $N_m(A_i)$, as $N_m$ is constant over the equivalence class, since the number of $m$-gons depends only on the unrooted diagram. So, we have
\[\expect{N_m} = \sum_{i=1}^kN_m(A_i)P([A_i]) + \condexpect{N_m}{S_n} P(S_n) \]
Subsituting in with Lemma \ref{fmandnm}, we have
\begin{equation}
\label{splitnm}
\expect{N_m} = \sum_{i=1}^k (4n \cdot \condexpect{F_m}{[A_i]})P([A_i]) + \condexpect{N_m}{S_n} P(S_n) 
\end{equation}
Now, we turn to the expectation of $F_m$, which is known by Lemma~\ref{fmexpectation} and related to $N_m$ by Lemma~\ref{fmandnm}. We decompose in the exact same way as with $N_m$:
\[\expect{F_m} = \sum_{i=1}^k\condexpect{F_m}{ [A_i]} P([A_i]) + \condexpect{F_m}{S_n} P(S_n) \]
Note that we get a similar term to before. Multiplying through by $4n$, and subtracting the rightmost term gives
\[ 4n \cdot \expect{F_m}-4n \cdot \condexpect{F_m}{S_n} P(S_n) = \sum_{i=1}^k 4n \cdot \condexpect{F_m }{ [A_i]} P([A_i])  \]
Now, we can substitute directly back into \ref{splitnm} above, and obtain
\[\expect{N_m} = (4n \cdot \expect{F_m}) - (4n \cdot \condexpect{F_m}{S_n} P(S_n)) + \condexpect{N_m}{S_n} P(S_n) \]

We know from \cite{richmondwormald} that the probability of selecting a symmetric diagram is exponentially small as $n \rightarrow \infty$, so $P(S_n) = O(a^n)$, $ 0 < a < 1$. The conditional expectation of $F_m$ over $S_n$ is bounded (at most 1), and the conditional expectation of $N_m$ is at most linear in $n$, as the highest it can be is bounded by the total number of faces, which is $n+2$ by a simple Euler characteristic argument. Hence, by Lemma \ref{fmexpectation} and these facts, we have the desired result of
\[ \expect{N_m} = \frac{4n}{m} P(n,m) + O(na^{n}) \]

\end{proof}

So, to find explicitly the expected value of $N_m$, it remains to get a formula for $P(n,m)$.  In fact, the only case we need, to compute the expected number of bigons, is $P(n,2)$, which we can compute explicitly and directly from the definitions.
\begin{lem}
The limiting behavior of $P(n,2)$ is given by
\[ P(n,2) = \frac{4}{27} + \frac{10}{27n}+O \left( \frac{1}{n^2} \right) \]
\end{lem}
\begin{proof}
By the enumeration of Brown in Theorem \ref{brownenumeration},
\[ |SQ(n,2)| = \frac{2}{(2n-2)!} \left( \frac{(3n-5)!}{(n-2)!} - 2\frac{(3n-6)!}{(n-3)!} -3\frac{(3n-7)!}{(n-4)!} \right) \]
From here, we can factor out the factorials
\[ |SQ(n,2)| = \frac{2}{(2n-2)!} \frac{(3n-7)!}{(n-4)!} \left( \frac{(3n-5)(3n-6)}{(n-2)(n-3)} - 2\frac{(3n-6)}{(n-3)} -3 \right) \]
Dividing now by $|SQ(n)|$ and simplifying the right factor, we can pair all the factorials so that we are left with just a rational function in $n$
\[ \frac{|SQ(n,2)|}{|SQ(n)|} = \frac{(2n-1)(n)(n-1)(n-2)(n-3)}{(3n-3)(3n-4)(3n-5)(3n-6)} \left(\frac{6}{n-3} \right)\]
So simplifying and expanding the first two terms of the Laurent series about $\infty$ gives
\[ P(n,2) = \frac{4}{27} + \frac{10}{27n} + O\left( \frac{1}{n^2} \right)\]
which is the formula claimed.
\end{proof}

This gives the expected twist number of a diagram is $t(D) \approx (1-8/27)n = 19n/27$ for large $n$. For larger $m$, we can numerically approximate with large values for $n$ the expected portion of $m$-gons.

\begin{center}
 \begin{tabular}{||c | c||} 
 \hline
 $m$ & $4/m \cdot P(1000,m)$  \\ [0.5ex] 
 \hline\hline
 2 & 0.2964445 \\ 
 \hline
 3 & 0.2415913  \\
 \hline
 4 & 0.1643246  \\
 \hline
 5 & 0.1068911  \\
 \hline
 6 & 0.0686226 \\ 
\hline
 7 & 0.0439052 \\ [1ex] 
 \hline
\end{tabular}
\end{center}

Now that we have asymptotics for the twist number, using the bounds in \eqref{lackenbyineq}, we can get the asymptotic estimates on the expectation of the volume, and prove the first theorem of the introduction, repeated here.
\theoremstyle{plain}
\newtheorem*{thm:expectedvolume}{Theorem \ref{thm:expectedvolume}}
\begin{thm:expectedvolume}
Let $\Vol$ be the random variable which returns the hyperbolic volume of a random alternating link diagram with $n$ crossings. The expected hyperbolic volume is bounded by
\begin{equation*}
\left( \frac{19v_3}{54} \right) n + \left( \frac{10v_3}{27}-2 \right) +O\left( \frac{1}{n} \right) \leq \expect{\Vol}  <  \left( \frac{190v_3}{27} \right) n + \left( \frac{200v_3}{27}-1 \right) + O\left( \frac{1}{n}\right) 
\end{equation*}
Numerically, the coefficients on $n$ of the lower bound and upper bounds are approximately 0.35711 and 7.14217, respectively.
\end{thm:expectedvolume}

\begin{proof}
Note first that the inequality in \eqref{lackenbyineq} only apply to hyperbolic diagrams. So, if $H$ is the event that the diagram is hyperbolic, then
\[\expect{\Vol} = \condexpect{\Vol}{H} P(H) + 0  \]
as the volume is 0 in the other case.

Which diagrams in this model are not hyperbolic? As mentioned before, due to \cite{menasco84}, any such diagram must represent a torus link; however, there are just two such diagrams, the two rootings of the standard diagram of the 2-braid torus link $T(2,n)$ \cite{adams}, as in Figure \ref{fig:toruslink}.
For hyperbolic diagrams, we have the fact that the twist number $t$ is the number of crossings $n$ minus the number of bigons $N_2$, so we can compute the expected twist number for hyperbolic diagrams. For the two exceptions of $T(2,n)$, we have $n$ crossings, $n$ bigons, and twist number 1. That is, the conditional expectation for $N_2$ is 
\[\expect{N_2} = \condexpect{N_2}{ H} P(H) + n (P(H^c))  \]
In the hyperbolic case, we can compute the expected twist number using Theorem~\ref{expectedbigons}.

 \begin{align*}
\condexpect{t}{H} & = n - \condexpect{N_2}{H}  \\
                  & = n - ( \expect{N_2} - nP(H^c) ) \frac{1}{P(H)}  \\
                  & = n - (2n P(2,n) + O(na^n) - nP(H^c)) \frac{1}{P(H)} \\
                  & = n - \left( 2n \left( \frac{4}{27} + \frac{10}{27n} + O\left( \frac{1}{n^2}\right) \right) + O( na^n ) - n P(H^c) \right) \frac{1}{P(H)} \\
\intertext{Expanding out everything, and combining the asymptotically small terms gives}
                  & = \left( nP(H) - \frac{8n}{27} + \frac{20}{27} + O \left( \frac{1}{n} \right) \right) \frac{1}{P(H)} \\
                  & = \left( n(1-P(H^c)) - \frac{8n}{27} + \frac{20}{27} + O\left( \frac{1}{n} \right) \right) \frac{1}{P(H)} \\
                  & = \left( \frac{19n}{27} -nP(H^c) + \frac{20}{27} + O \left( \frac{1}{n}\right) \right) \frac{1}{P(H)} \\
\intertext{Absorbing the exponentially small term $nP(H^c)$ into $O \left( \frac{1}{n} \right)$, we get}
                  & = \left( \frac{19n}{27} + \frac{20}{27} + O\left( \frac{1}{n} \right) \right) \frac{1}{P(H)} 
 \end{align*}

With this, we are finally set up to bound the expected volume.  For hyperbolic diagrams, we have by linearity and monotonicity of expectation, as well as \eqref{lackenbyineq}

\[ \frac{v_3}{2}(t-2)  \leq  \Vol   <  10v_3(t-1) \]
\[ \frac{v_3}{2}(\condexpect{t}{H}-2)  \leq  \condexpect{\Vol}{H}   <  10v_3(\condexpect{t}{H}-1) \]

Multiplying through by $P(H)$, we get the expected volume over all diagrams, and with the result of the computation of the expected twist number, we get

\[ \frac{v_3}{2} \left( \frac{19n}{27} + \frac{20}{27} + O \left( \frac{1}{n} \right) -2P(H) \right)  \leq  \expect{\Vol}   <  10v_3 \left( \frac{19n}{27} + \frac{20}{27} + O \left( \frac{1}{n} \right) -P(H) \right) \]
\[ \frac{19v_3}{54}n + \frac{10v_3}{27}-2 + O\left( \frac{1}{n} \right)  \leq  \expect{\Vol}   <  \frac{190v_3}{27}n + \frac{200v_3}{27}-1 + O\left( \frac{1}{n} \right) \]

Again, we have replaced the $P(H)$ terms with $1-P(H^c)$ and absorbed those exponentially small terms into the $O(\frac{1}{n})$. So, the result is obtained.

\end{proof}

So, the expected volume is bounded below and above by linear functions, up to some terms that go to zero for large $n$.  The upper bound we get from this process is worse than the bound we would have gotten from using the octahedral bound. For an $n$ crossing link diagram representing a link $K$, this bound gives that
\[ \mathit{Vol}(S^3-K) \leq v_8 n\]
This bound applies for the nonhyperbolic cases trivially, so by taking expectation, we get the following bound.
\begin{lem}
  The expected hyperbolic volume $\Vol$ of a random alternating link diagram with $n$ crossings is bounded above:
\begin{equation}
\label{octahedralbound}
\expect{\Vol} \leq v_8 n   
\end{equation}
where $v_8 \approx 3.6638$ is the volume of a regular hyperbolic ideal octahedron.
\end{lem}

Hence, for a slightly stronger result, we can use the lower bound from the proof above and the upper bound from the octahedral decomposition, giving the statement of Theorem~\ref{thm:expectedvolume} in the introduction.
\begin{thm:expectedvolume}[Improved]
The expected hyperbolic volume $\expect{\Vol}$ of an $n$-crossing alternating link diagram is bounded by
\begin{equation*}
\left( \frac{19v_3}{54} \right) n + \left( \frac{10v_3}{27}-2 \right) +O\left( \frac{1}{n} \right) \leq \expect{\Vol} \leq v_8 n
\end{equation*}
Numerically, the coefficients on $n$ of the lower bound and upper bounds are approximately 0.3571 and 3.6638, respectively.
\end{thm:expectedvolume}

We'd like to divide by $n$ and take a limit, but it's unknown whether the expected volume per crossing converges to any limiting value. Accordingly, we offer this seemingly intractable question as a conjecture. We also present experimental evidence for it in the next section.
\begin{conj}
The expectation of the hyperbolic volume per crossing,
\[ \frac{1}{n}\expect{\Vol} \]
 converges to a limiting value as $n \to \infty$.
\end{conj}
However, we can still state limiting behavior in terms of $\liminf$ and $\limsup$.

\begin{cor}
Let $\Vol$ be the random variable which returns the hyperbolic volume of a random alternating link diagram with $n$ crossings. The expected hyperbolic volume per crossing is bounded by
\begin{equation*}
\frac{19v_3}{54} \leq \liminf_{n \to \infty} \left( \frac{1}{n}\expect{\Vol} \right)  \leq  \limsup_{n \to \infty} \left( \frac{1}{n}\expect{\Vol} \right)  \leq  v_8
\end{equation*}
\end{cor}
The conjecture then would be that the $\liminf$ and $\limsup$ above are actually equal.
\subsection{Numerical Experiments} \label{numericalexperiments}

The random link sampling algorithm is incorporated into SnapPy; we can look at the distribution of volumes that we get for different sizes of alternating diagrams.  Here we generated 1 million random diagrams of varying sizes, ranging from 10 to 1000 crossings, every 10 crossings, and plot in Figure \ref{fig:volumeboxplot}.

\begin{figure}
\centering
\includegraphics[scale=.5]{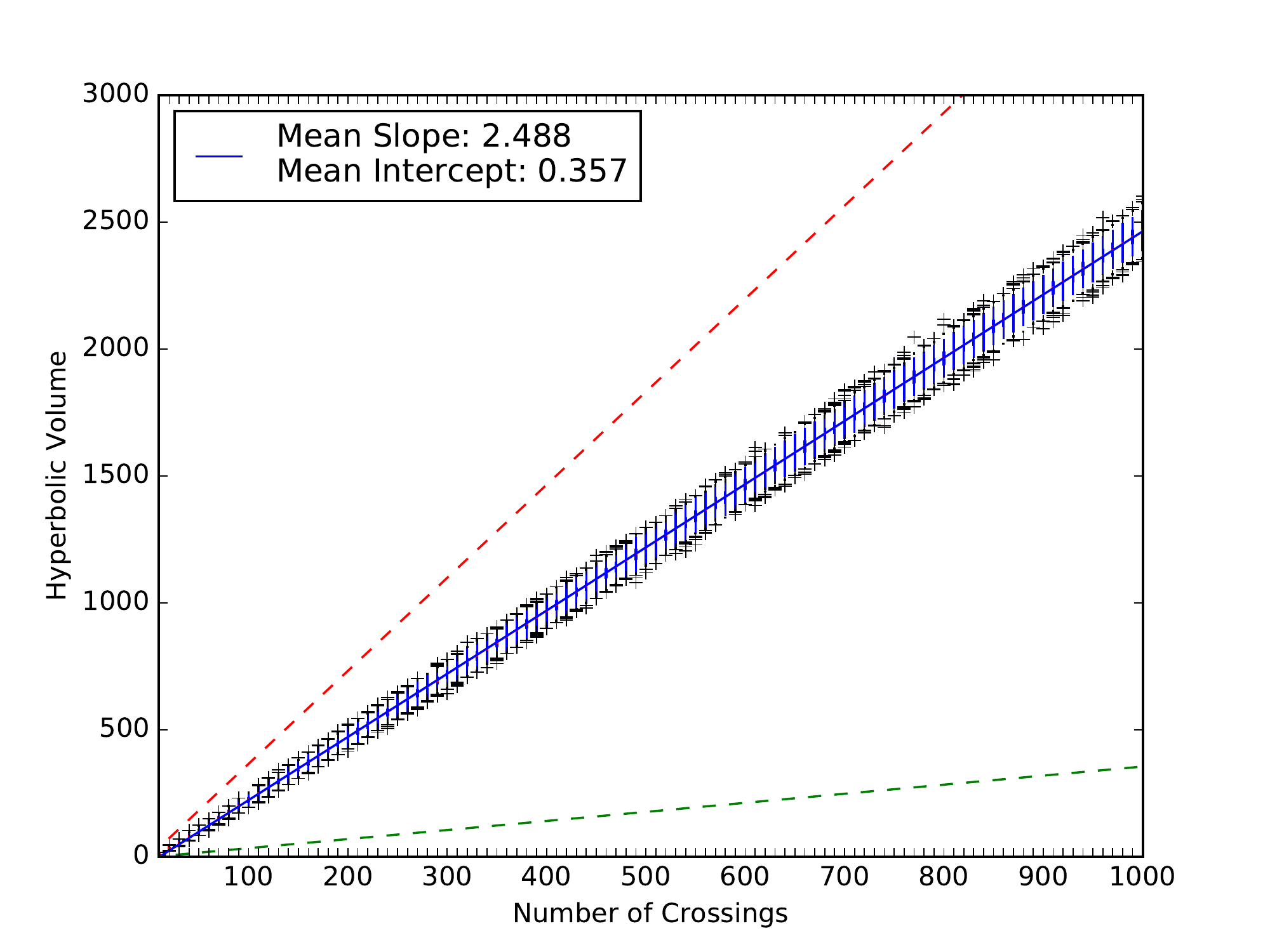}
\caption{Box-and-whisker plot of volumes of alternating links for a range of sizes. We generated 10,000 samples of size 10, 20, 30, etc. up to 1000. The dashed lines are the bounds from Theorem \ref{thm:expectedvolume}, and the plus signs are outliers.}
\label{fig:volumeboxplot}
\end{figure}

\begin{figure}
\centering
\includegraphics[scale=.5]{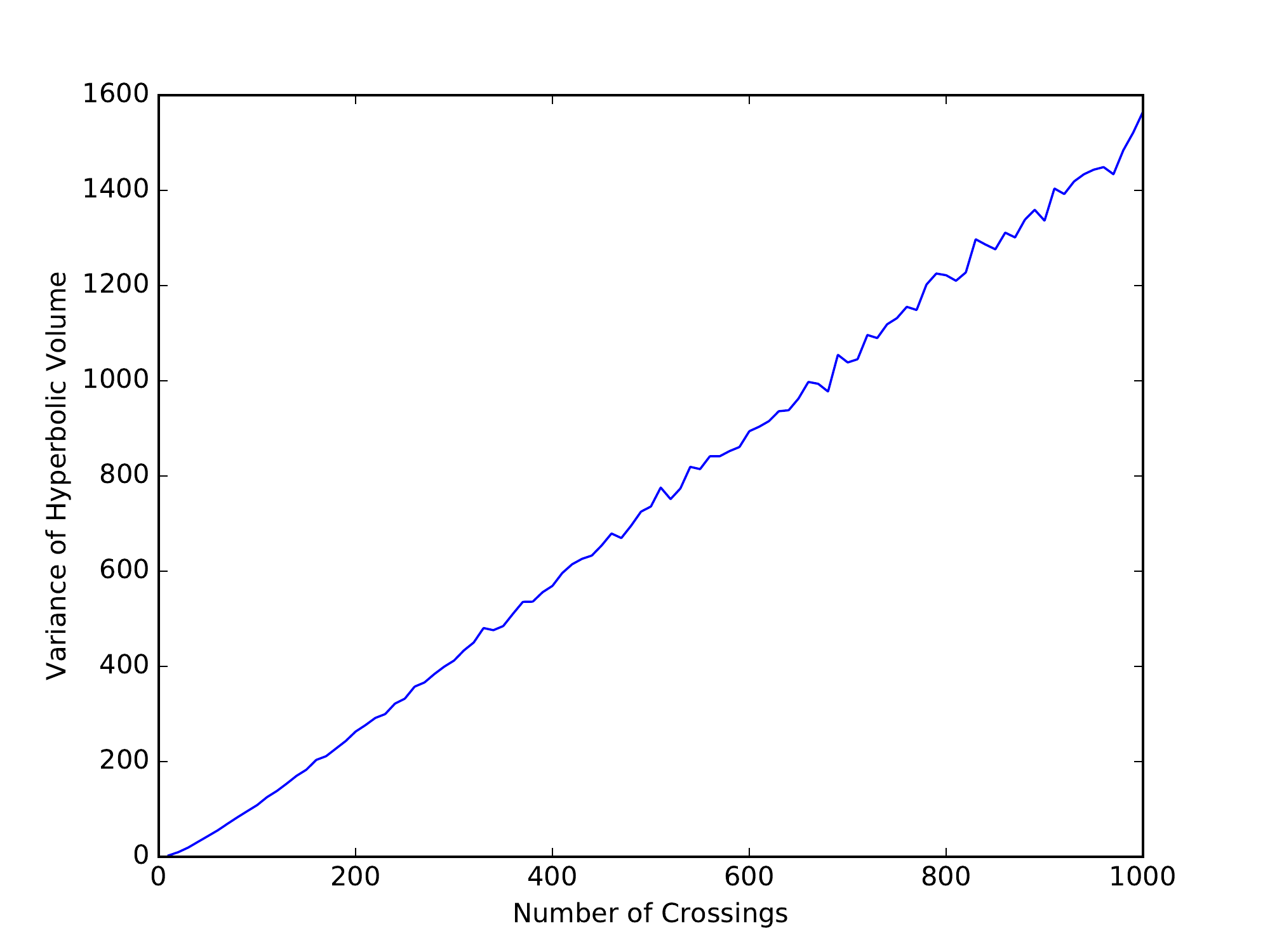}
\caption{Plot showing the growth of the variance of the data in Figure \ref{fig:volumeboxplot}. }
\label{fig:volume_variance}
\end{figure}

The volume does in fact appear to grow linearly, with slope around 2.5, which is then the observed expected volume per crossing. The variances are plotted in Figure \ref{fig:volume_variance}, and appear to grow linearly as well. We can also look at the distribution for a given number of crossings, which we plot in Figure~\ref{fig:volumedist}.

\begin{figure}[b]
\centering
\includegraphics[scale=.6]{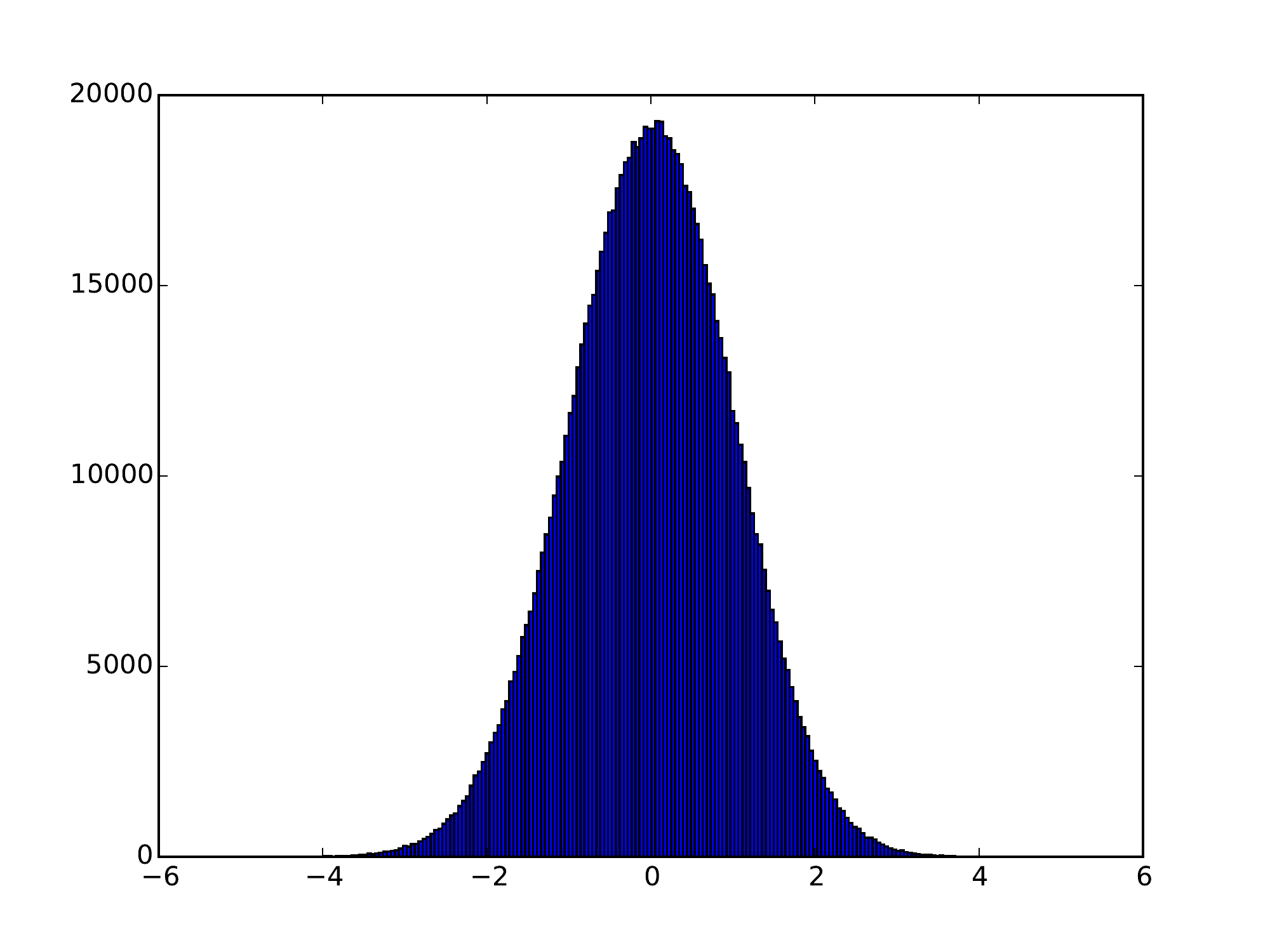}
\caption{Histogram for the volumes of 1 million randomly sampled alternating links with 900 crossings, normalized to mean 0 and variance 1.  The actual mean and variance were approximately 2216.46 and 1371.02 respectively.}
\label{fig:volumedist}
\end{figure}

Normalizing, it seems to have a limiting distribution, though the distribution appears to not be normal --- experimentally, the distribution skews to the left. For each number of crossings, we took a million samples of volumes of random alternating link diagrams and computed the skew, which appears to be consistently small but negative. A possible explanation for this is that the distribution for the number of bigons (when normalized) appears to be slightly skewed in the other direction, positively (Figure \ref{fig:vol_and_bigon_skews}). 

\begin{center}
 \begin{tabular}{c |  c | c } 
% \hline
 Number of Crossings & Skewness of Volume & Skewness of Number of Bigons  \\ [0.5ex] 
 \hline
 100 & -0.093858033849 & 0.040877113357 \\ 
% \hline
 300 & -0.049351483215 & 0.018498451862 \\
% \hline
 500 & -0.048770255118 & 0.018192743976 \\
% \hline
 700 & -0.050229963233 & 0.014187648283 \\
% \hline
 900 & -0.052681105543 & 0.011986807300 \\
% \hline
 1100 & -0.051342016899 & 0.015350076951\\ [1ex] 
% \hline
\end{tabular}
\end{center}

In fact, we can look at some higher moments as well (see Figure \ref{fig:moments}) and note that those also appear to vaguely tend towards limiting values, though the data is sparser than the above. This could lead one to conjecture the following.

\begin{conj}
The normalized distribution of $\Vol$ in this model converges to a skewed limiting distribution as $n \to \infty$.
\end{conj}

\begin{figure}%[!h]
\centering
\includegraphics[scale=.5]{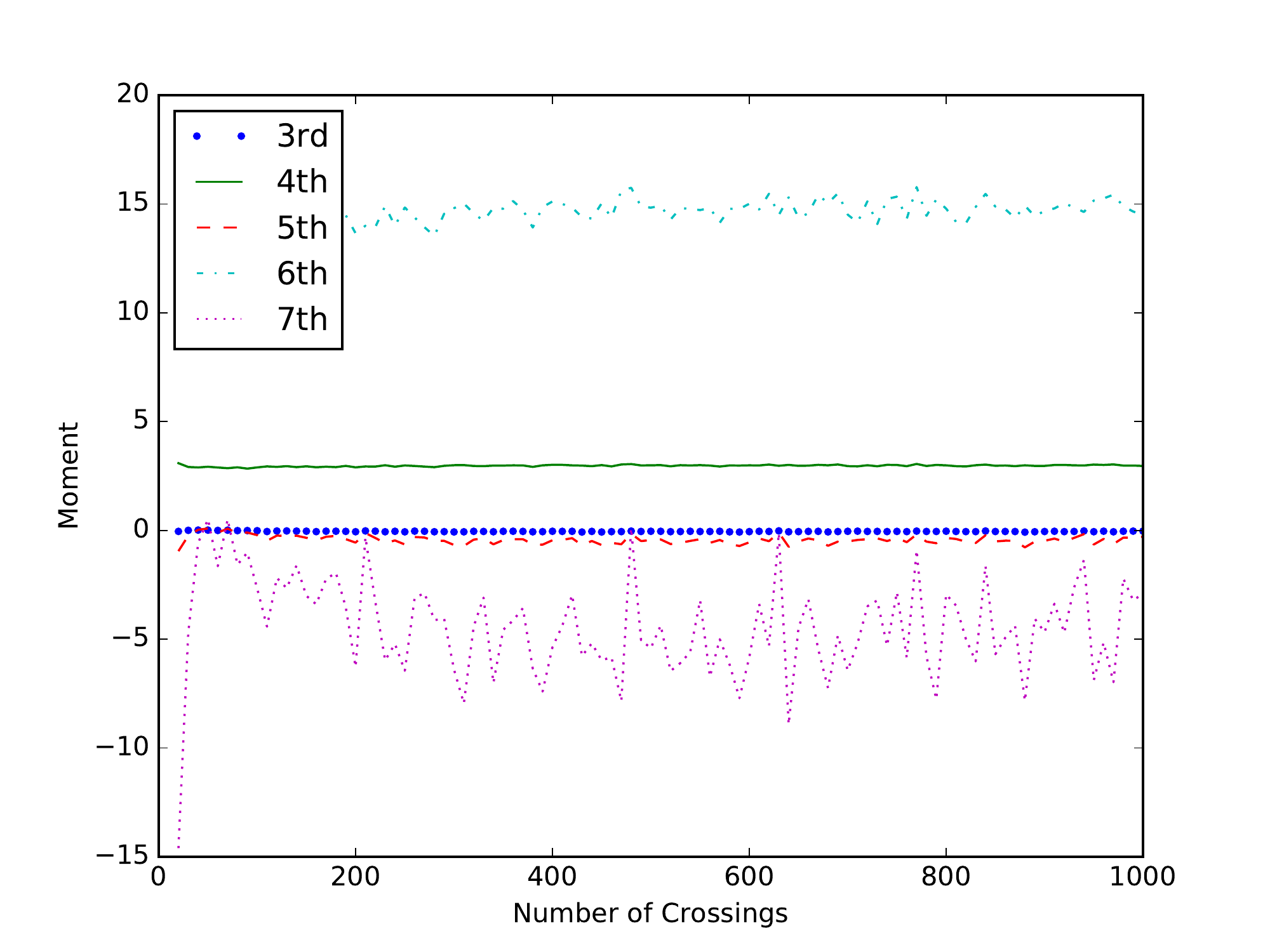}
\caption{Some higher moments for alternating link diagrams sampled from crossing number 20 to 1000.}
\label{fig:moments}
\end{figure}

\begin{figure}%[!h]
\centering
\includegraphics[scale=.5]{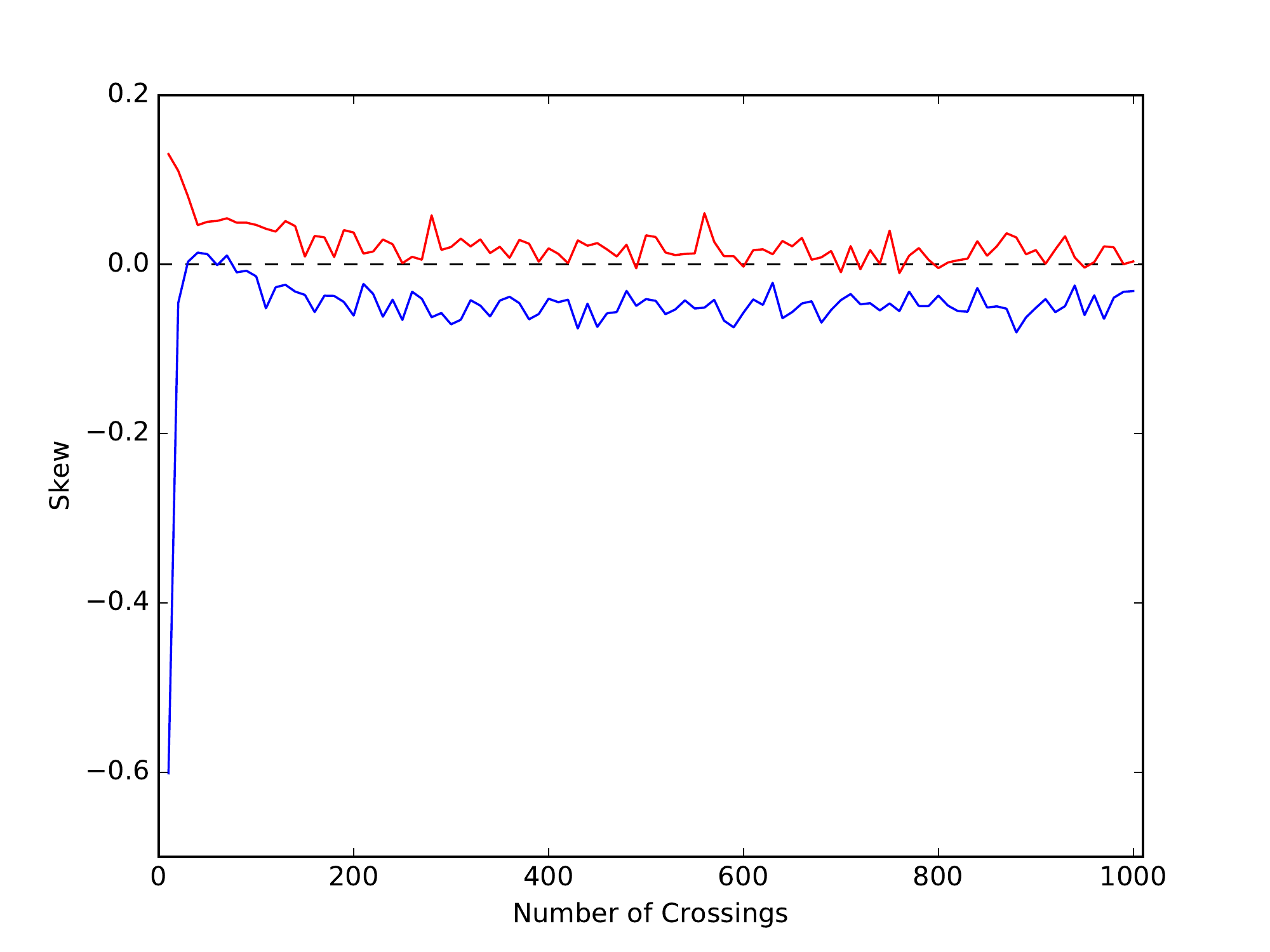}
\caption{The skews for the volume and number of bigons for collections of samples ranging from 10 crossings to 1000. The skew for the volume (the lower graph) tends to be slightly negative, and for the number of bigons (the upper graph) slightly positive.}
\label{fig:vol_and_bigon_skews}
\end{figure}

\section{Nonalternating Diagrams} \label{nonalternatingdiagrams}
If we expand to randomly generating nonalternating diagrams, then large diagrams actually are increasingly \emph{non-hyperbolic}; the existence of certain ``local pictures'' forbids the link from admiting a hyperbolic structure. This is due to the existence of essential tori.

\begin{figure}
\begin{subfigure}{.5\textwidth}
\centering
\includegraphics[width=.9\linewidth]{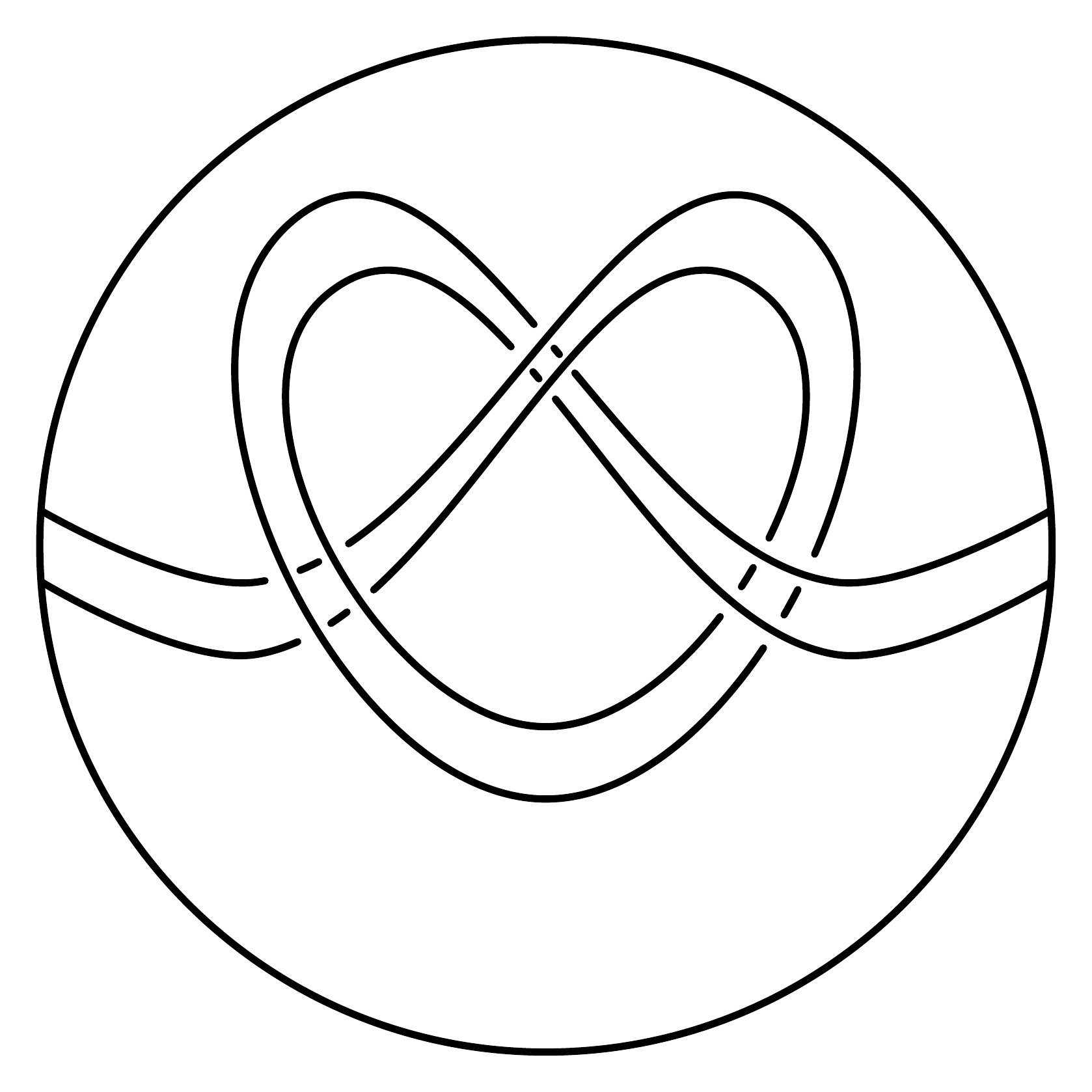}
\end{subfigure}%
\begin{subfigure}{.5\textwidth}
\centering
\includegraphics[width=.9\linewidth]{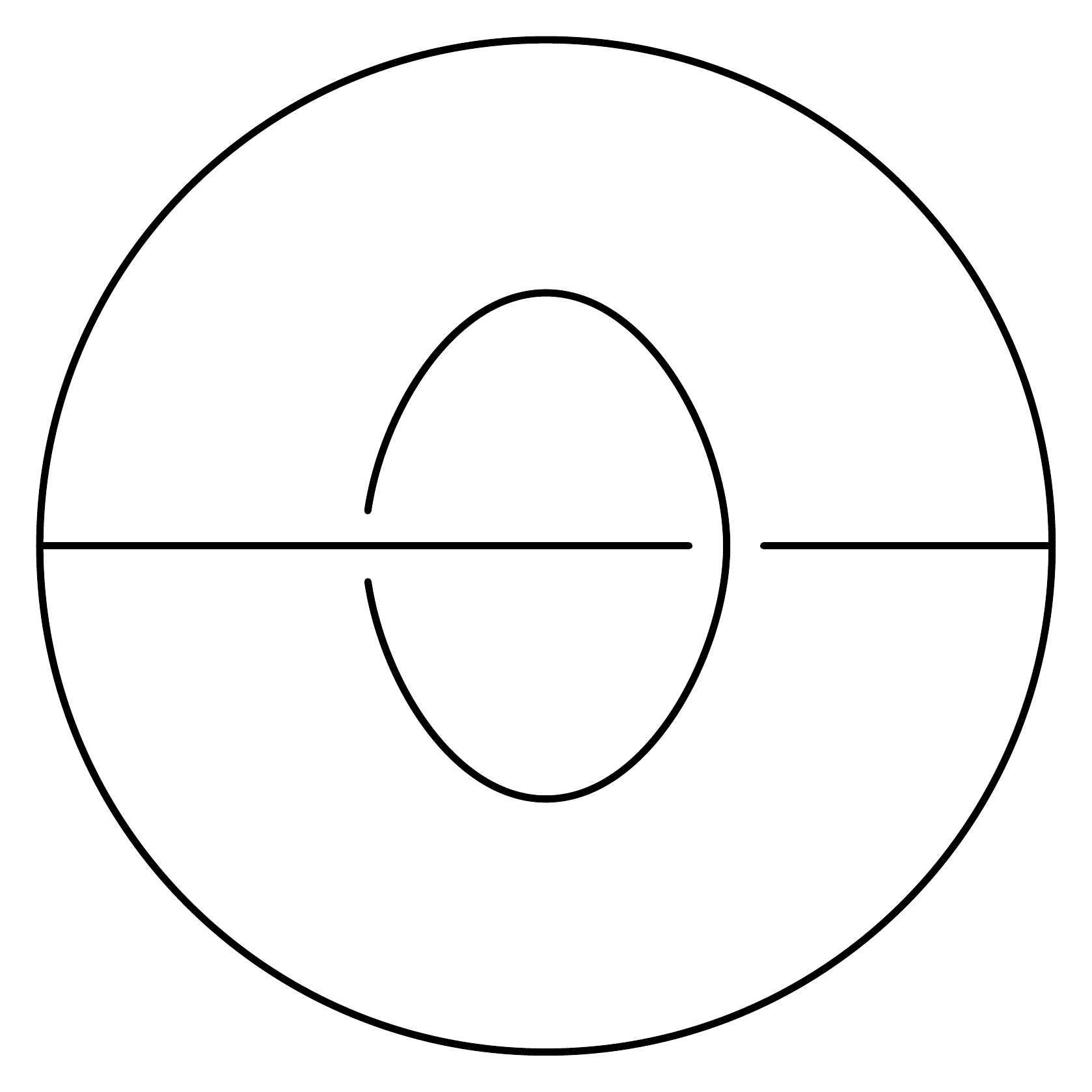}
\end{subfigure}
\caption{}
\label{fig:forbiddentangle}
\end{figure}

A diagram where either of the configurations in Figure \ref{fig:forbiddentangle} is encountered is a satellite link, when the tangle on the other side is nontrivial. We would like to compute the probability that a given tangle occurs in a diagram. Since, for nonalternating links, the diagram being 3-edge-connected doesn't assure primeness of the link itself, we will simplify the matter by considering \emph{all} rooted 4-valent planar maps with $n$ vertices. This set is in bijection with the set of \emph{quadrangulations} with $n$ \emph{faces}, which we denote $Q(n)$. A quadrangulation is simply a planar map in which every face has four sides. This set $Q(n)$ was enumerated by Tutte in \cite{tuttecensus}:
\[|Q(n)| = \frac{2(3^n)(2n)!}{n!(n+2)!} \]
 Combinatorially, we will view a tangle as being dual to a \emph{quadrangulation with boundary}: a planar map in which every face has 4 sides except for one face, the boundary face, which has arbitrary even size. We will call the number of faces the \emph{area} and the length of the boundary face the \emph{perimeter}. We will also restrict to quadrangulations with \emph{self-avoiding boundary} --- walking around the boundary face, we encounter each vertex only once (see Figure \ref{fig:quadrangulations}). In the dual, we can think of this as the projection of a tangle in which there are no strands which cross no other strands.

\begin{figure}
\centering
\includegraphics[scale=.3]{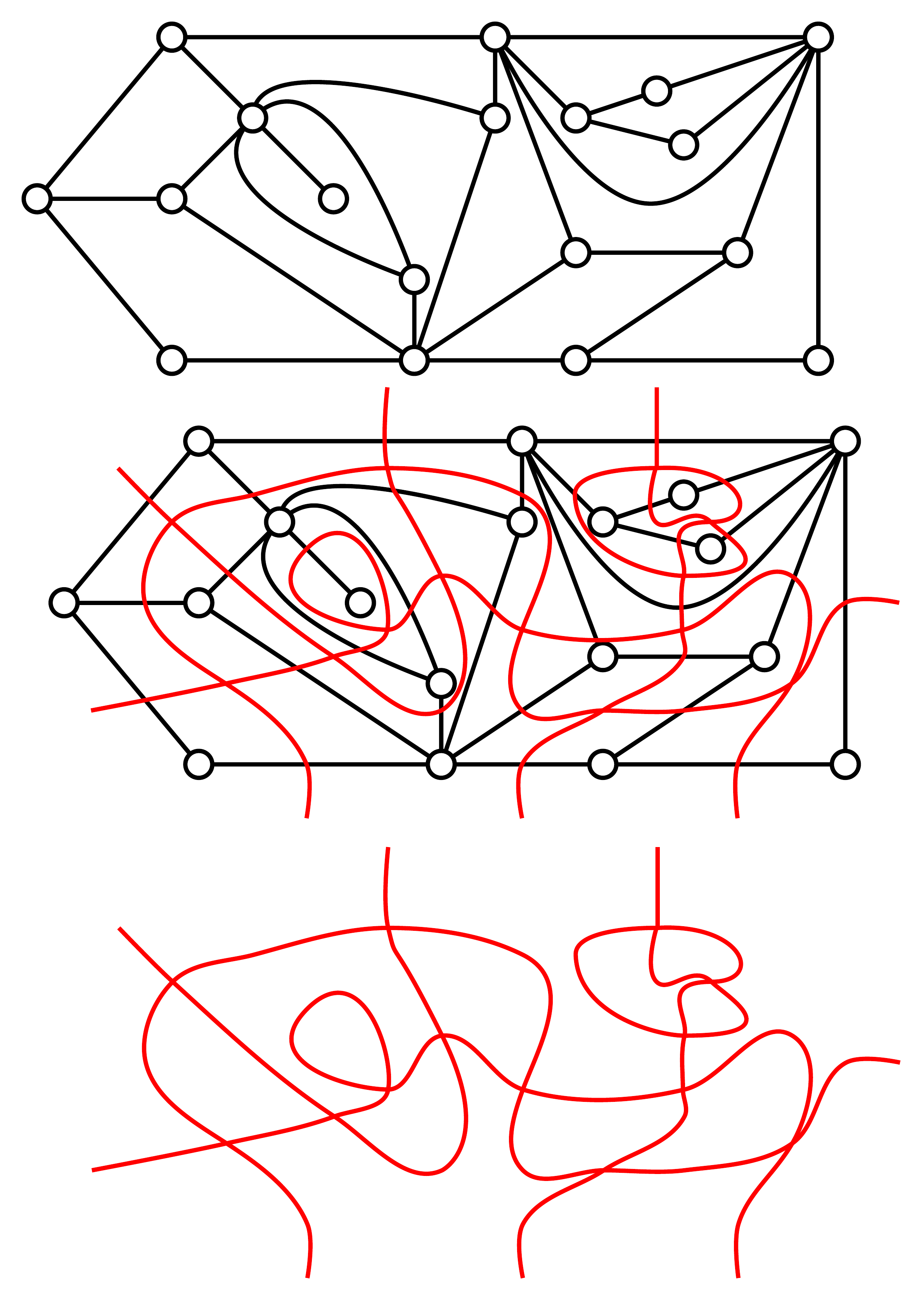}
\caption{A quadrangulation with self-avoiding boundary and the local picture defined by the dual graph. Here the quadrangulation has area 13 and perimeter 8.}
\label{fig:quadrangulations}
\end{figure}

Let $Q(n,p)$ then be the set of all rooted quadrangulations with self-avoiding boundary, area $n$, and perimeter $2p$, such that boundary face is to the right of the root. These were enumerated in \cite{bouttierguitter}:

\[|Q(n,p)| = 3^{n-p} \frac{(3p)!}{p!(2p-1)!} \frac{(2n+p-1)!}{(n-p+1)!(n+2p)!} \]

We want to know the probability that a given rooted quadrangulation with boundary $K$ has an embedding into a quadrangulation $Q$, in terms of the area and perimeter of $K$ and the area of $Q$. We require as well that the embedding take root to root. The probability is computed with a counting argument analogous to one used by Krikun in \cite{krikun}; the only difference here is the use of a slightly more general class of quadrangulation with boundary which results in a slightly simpler formula.

\begin{prop}
\label{prop:nonalternatingprob}
The probability that a rooted quadrangulation $K$ with area $n$ and perimeter $2p$ embeds in a rooted quadrangulation $Q$ with $N$ faces depends only on $n,p, $ and $N$:

\[\mathit{Prob}(K \hookrightarrow Q) = P(n,p,N) := \frac{|Q(N-n,p)|}{|Q(N)|} \]

As $N \rightarrow \infty$ the limiting probability is

\[ P_{\mathit{lim}}(n,p) := \lim_{N \rightarrow \infty } P(n,p,N) = \frac{(3p)!}{9\cdot p! (2p-1)!} \left( \frac{2}{3} \right)^{p-2} 12^{-n} \]
\end{prop}
\begin{proof}
As in \cite{krikun}, we count the number of quadrangulations $Q$ in which $K$ embeds. Let $Q$ be such a quadrangulation. As is, the root of $K$ is not along the boundary as in Bouttier and Guitter's definition in \cite{bouttierguitter}. To fix this, choose, in some deterministic fashion, a root along the boundary of $K$, oriented with the boundary face to the right. Then, cut $Q$ along the boundary of $K$; outside $K$ we now have a quadrangulation with boundary with perimeter $p$ and area $N-n$, and a root in a fixed orientation along the boundary. So, given $Q$ and $K$, we get an element of $Q(N-n,p)$. Conversely, given $K$ and an element of $Q(N-n,p)$, we can simply glue along the boundaries, attaching the root edges, to uniquely give an inverse. Then, portion of quadrangulations $Q \in Q(N)$ in which $K$ embeds is 

\[\frac{|Q(N-n,p)|}{|Q(N)|} \]

To get the explicit formula in the limiting case, we use Stirling's approximation to get asymptotics for $Q(N,p)$ and $Q(N)$ for large $N$:

\begin{equation*}
\begin{split}
|Q(n,p)| & =  3^{n-p} \frac{(3p)!}{p!(2p-1)!} \frac{(2n+p-1)!}{(n-p+1)!(n+2p)!} \\
         & \approx  \frac{(3p)!2^{p-1}}{3^p \cdot p!(2p-1)!\sqrt{\pi}}12^k k^{-5/2} \\
P(n,p,N) & =  \frac{|Q(N-n,p)|}{|Q(N)|} \\
         & \approx  \frac{(3p)!2^{p-2}12^{-n}}{3^p \cdot p! (2p-1)!} \frac{(N-n)^{-5/2}}{N^{-5/2}} \\
         & \rightarrow  \frac{(3p)!}{9\cdot p! (2p-1)!} \left( \frac{2}{3} \right)^{p-2} 12^{-n} 
\end{split}
\end{equation*}
\end{proof}

So, by taking the dual picture, any possible neighborhood, including the underlying 4-valent map of the tangle in Figure \ref{fig:forbiddentangle}, occurs with positive probability, exponential decreasing in the number of crossings in the tangle. Since the crossings are decided by independent coin flips, this means that the probability that a tangle $T$ with $n$ crossings and $2p$ boundary points appears around the root in a random rooted link diagram $L$ with $N$ crossings will be given by:
\[ \mathit{Prob}(T \hookrightarrow L) = (2^{-n})P(n,p,N) \]
This shows that large diagrams in this model will be non-hyperbolic with probability approaching 1. More specifically, one can show the following.

\newtheorem*{thm:genericallynothyperbolic}{Theorem \ref{thm:genericallynothyperbolic}}
\begin{thm:genericallynothyperbolic}
Let $T$ be a rooted tangle with $n$ crossings and $2p$ boundary points. Then, the number of rootings $N_{T,c}$ of a random $c$-crossing link diagram $L$ for which $T$ embeds around the root has expectation which is asymptotically linear in $c$, the number of crossings:
\[  \expect{N_{T,c}} = (4c) 2^{-n} P(n,p,c) + O(cb^{-c}) \]
for some $b \in (0,1)$. Normalizing by the number of crossings $c$, we have a positive limiting expectation:
\[ \lim_{c \to \infty} \frac{1}{c} \expect{N_{T,c}} = 4\cdot2^{-n} P_{\mathit{lim}}(n,p) > 0 \]
\end{thm:genericallynothyperbolic}
\begin{proof}
Recall that the probability $P(n,p,c)$ is bounded away from 0 for large $c$, so such a function of $c$ is in fact asymptotically linear. We will work with the associated quadrangulation with boundary $K$ of $T$ and the associated quadrangulation $D$ of $L$. The proof proceeds similarly to the proof that the expected number of bigons is linear in the number of crossings. We define two random variables on the set $Q(c)$ of quadrangulations with $c$ faces:
\begin{align*}
 N_K(D) &= \text{\# of rootings of $D$ for which $K$ embeds around the root} \\ 
 R_K(D) &= \begin{cases}
    1 & \mbox{if $K$ embeds around the root of $D$} \\
    0 & \mbox{otherwise}
   \end{cases} \\ 
\end{align*}
We are suppressing the number of crossings $c$ in the subscripts $N_{K,c}$ and $R_{K,c}$ for ease of notation.
As in the proof for the alternating case, we can find the expectation of $R_K$ using Proposition~\ref{prop:nonalternatingprob}, and relate that to the expectation of $N_K$. The expectation of $R_K$ is by definition the probability that $K$ embeds around the root of $D$.
\begin{equation}
\label{totalRKexpect}
\expect{R_K} = \mathit{Prob}(K \hookrightarrow D) = P(n,p,c)
\end{equation}
Exactly as before, we partition $Q(c)$ into equivalence classes of diagrams under unrooted equivalence, and choose representatives $A_1, A_2, ... A_k$ of the asymmetric quadrangulations and let $S$ be the set of all symmetric quadrangulations. We'll denote the equivalence class of $D$ by $[D]$, which is the collection of all rootings of $D$. The relationship between the two random variables is then made clear by looking at the conditional expectations over an asymmetric equivalence class. Because $N_K$ only depends on the unrooted diagram,
\begin{equation}
\label{condNKexpect}
\condexpect{N_K}{ [A_i] } = N_K(A_i)
\end{equation}
For $R_K$, we sum $R_K(A_i)$ over all of the $4c$ rootings of the diagram $A_i$, which exactly counts the number of rootings around which $K$ embeds. Dividing by the total number of rootings $4c$,
\begin{equation}
\label{condRKexpect}
\condexpect{R_K}{ [A_i] } = \frac{1}{4c}N_K(A_i)
\end{equation}
Then, using the law of total expectation for the expectations of both random variables, we get:
 \begin{align*}
\expect{R_K} &= \sum_{i=1}^k\condexpect{R_K}{[A_i]}P([A_i]) + \condexpect{R_K}{S} P(S)  \\
\intertext{By \ref{condRKexpect},}
\expect{R_K} &= \sum_{i=1}^k\frac{1}{4c}N_K(A_i)P([A_i]) + \condexpect{R_K}{S} P(S) \\
\intertext{For $N_K$,}
\expect{N_K} &= \sum_{i=1}^k\condexpect{N_K}{[A_i]}P([A_i]) + \condexpect{N_K}{S} P(S) \\
\intertext{Using \ref{condNKexpect}, note that we have a very similar term}
\expect{N_K} &= \sum_{i=1}^k N_K(A_i)P([A_i]) + \condexpect{N_K}{S} P(S) \\
\intertext{Multiplying through the expectation of $R_K$ by $4c$ and substituting,}
\expect{N_K} &= (4c) \expect{R_K} - (4c)\condexpect{R_K}{S} P(S) + \condexpect{N_K}{S} P(S) \\
\intertext{Finally, by \ref{totalRKexpect} and \cite{richmondwormald},}
\expect{N_K} &= (4c) P(n,p,c) + O(c b^{-c})\\
 \end{align*}
 By taking the dual, we get the same formula for 4-valent planar maps, and since we are choosing uniformly how to orient the crossings, we get the result for rooted tangles embedding in rooted link diagrams by multiplying by $2^{-n}$, and the desired result is obtained.
 \end{proof}

Knowing that any local picture occurs with positive probability in a large random link immediately tells you a number of facts about random links.
\newtheorem*{cor:genericallynothyperbolic}{Corollary \ref{cor:genericallynothyperbolic}}
\begin{cor:genericallynothyperbolic}
For a random link diagram, the following quantities have expectation which is asymptotically bounded above and below by linear functions of the number of crossings:
\begin{enumerate}
\item The number of link components
\item The number of pieces in the connect sum decomposition
\item The number of pieces of the JSJ decomposition of the exterior
\item The Gromov norm of the exterior
\item The crossing number
\end{enumerate}
\end{cor:genericallynothyperbolic}
\begin{proof}
  All of these quantities have linear upper bounds in terms of the number of crossings of the link diagram. This is obvious for the first and last items; for the JSJ decomposition, we can triangulate the exterior of the link with a linearly many tetrahedra, and from there, there are only linearly many disjoint, nonparallel, incompressible surfaces in terms of the number of tetrahedra \cite{hass}. Hence, the number of pieces in the JSJ decomposition, which splits along tori, has a linear bound in terms of the number of crossings in the link diagram. The connect sum decomposition also splits along incompressible tori \cite{budney}, so this argument applies there as well.

In the case of the Gromov norm, again, we have a triangulation of the exterior with linearly many tetrahedra as a function of the number of crossings $n$; for example, the exterior can be divided into $4n+4$ tetrahedra \cite{weeks}. So, the Gromov norm is also bounded above by a linear function of $n$, as the Gromov norm is, by definition, an infimum over a set containing, for instance, the number of tetrahedra in a triangulation.

  To get a lower bound on the expectations, we apply Theorem~\ref{thm:genericallynothyperbolic} --- for each quantity, we take a local picture which adds some set amount to the total of that quantity when it occurs. For the number of pieces in the connect sum or JSJ decompositions, take the tangles in Figure~\ref{fig:forbiddentangle}. There are a linear number of each of these in expected value in a random link diagram, so the expected number of components is at least as large. The linearity of the expected Gromov norm follows immediately from this, as we have, in expectation, a linear number of connect summands which are figure 8 knots, for instance. Since, in the JSJ decomposition, the Gromov norm is (up to a constant factor) the sum of the volumes of the hyperbolic pieces \cite{budney}, we have a linear bound on the expectation of the Gromov norm.

Similarly, for the number of link components, we can use the Hopf link example from Figure~\ref{fig:forbiddentangle} again. Note that we even get linearly many split link components, as we could use that diagram but with one crossing changed, so that the loop is not linked anymore.

Finally, for the crossing number, if the expectation of the Gromov norm is bounded below by a linear function of the number of crossings, then so is the (minimal) crossing number --- this is an immediate consequence of the fact that the Gromov norm is bounded above by a linear function of the crossing number.
\end{proof}

%\begin{figure}
%\centering
%\includegraphics[scale=.6]{nonalternatingvolumes.eps}
%\caption{Box-and-whisker plot showing the Gromov norm (as computed by \emph{SnapPy} of randomly sampled
%connected link diagrams where over/under crossings are chosen independently. The plus signs denote outliers past 1.5 times the interquartile range of the lower or upper quartiles.}
%\label{fig:nonalternatingvolumes}
%\end{figure}

%\section{Xypic and diagrams}

%If you want to draw diagrams, you should use xypic.  It's actually
%much easier than it looks, and we've already included it in the header
%above.  Here is an example. 
%\[\xymatrix{
%FX \ar[r]^-{Ff} \ar[d]_{\eta_X} & FY \ar[d]^{\eta_Y} \\
%GX \ar[r]_-{Gf} & GY\\} \]

\section{Larger Faces} \label{largerfaces}
In the case of alternating diagrams, links with more bigons tended to have smaller volume, fixing a number of crossings, as it decreases the twist number. What effect is there from the other possible sizes of faces? For an alternating diagram $D$ with $n$ crossings, let the \emph{face type} of the diagram be the ordered list
\[(F_2(D),F_3(D), ..., F_{3n}(D)) \]
where $F_i(D)$ is the number of faces in $D$ with length $i$.  The cutoff of $3n$ is an easy upper bound for the maximum length of a face in a diagram using an Euler characteristic argument. Looking first at the Hoste-Thistlethwaite census of link exteriors with 14 crossings, we can sort lexicographically by the face type and plot against the volume (see Figure \ref{fig:Vol_vs_facetype_alt_and_nonalt_14_crossings}).

\begin{figure}
\centering
\includegraphics[scale=.5]{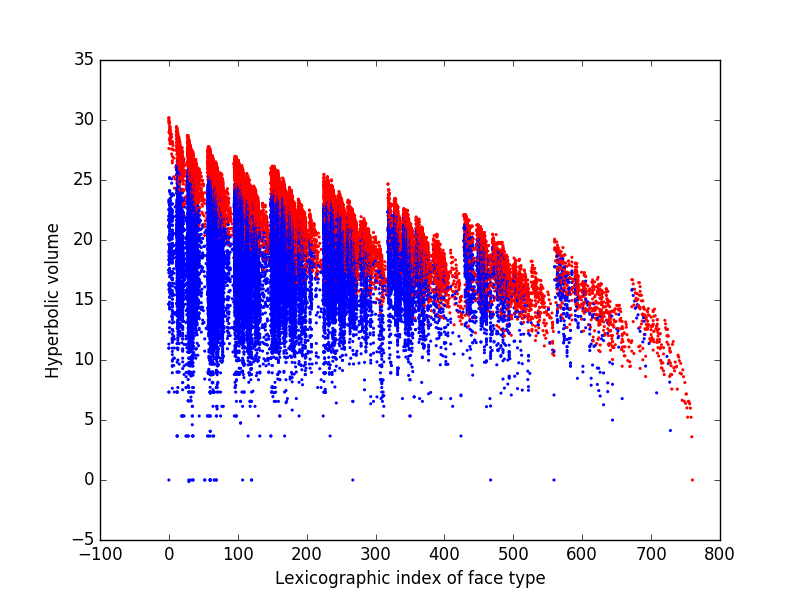}
\caption{}
\label{fig:Vol_vs_facetype_alt_and_nonalt_14_crossings}
\end{figure}

Each peak corresponds to incrementing the number of bigons. The general downward trend is because increasing the number of bigons decreases the twist number, so the volume tends to decrease.  If we look at randomly generated alternating links with 50 crossings we still see a similar pattern (see Figure \ref{fig:Vol_vs_facetype_alt_50}).

\begin{figure}[b]
\centering
\includegraphics[scale=.5]{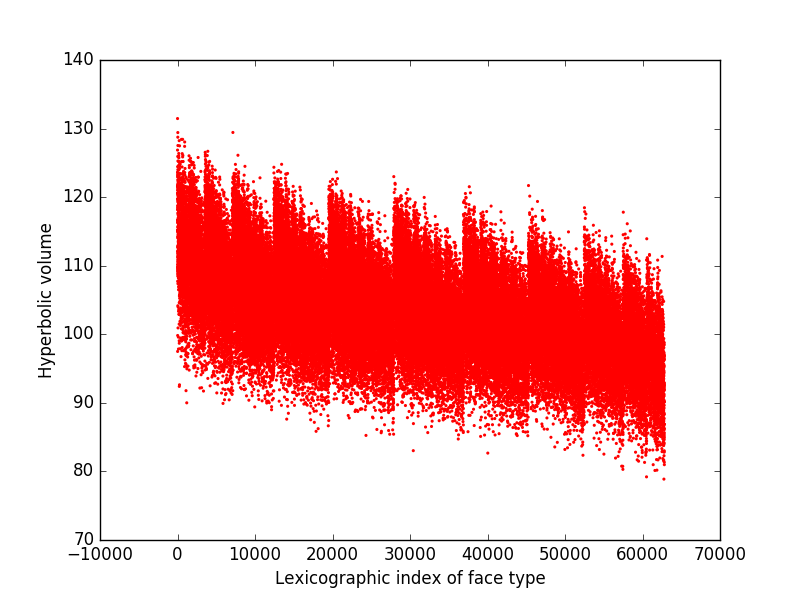}
\caption{}
\label{fig:Vol_vs_facetype_alt_50}
\end{figure}

Note that for a given peak, the volume tends to sharply decrease in the number of 3 sided faces that we see, so there appears to be some dependence on larger faces.  Zooming in on one of these peaks (fixing the number of bigons) gives Figure \ref{fig:Vol_vs_facetype_alt_50_zoomed}.

\begin{figure}[b]
\centering
\includegraphics[scale=.5]{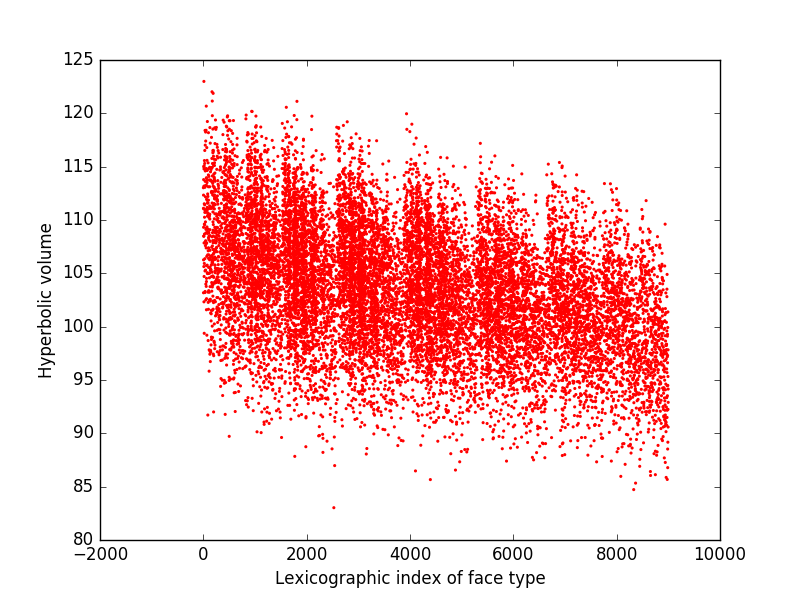}
\caption{}
\label{fig:Vol_vs_facetype_alt_50_zoomed}
\end{figure}

In \cite{garoufalidisle}, Garoufalidis and L\^{e} give a sequence of knot complements which asymptotically approach the maximal possible volume per crossing of $v_8$. These knots are formed by taking a ``weave'' and closing off the ends to form a knot, as in Figure~\ref{fig:weave}.

\begin{figure}
\centering
\includegraphics[scale=.43]{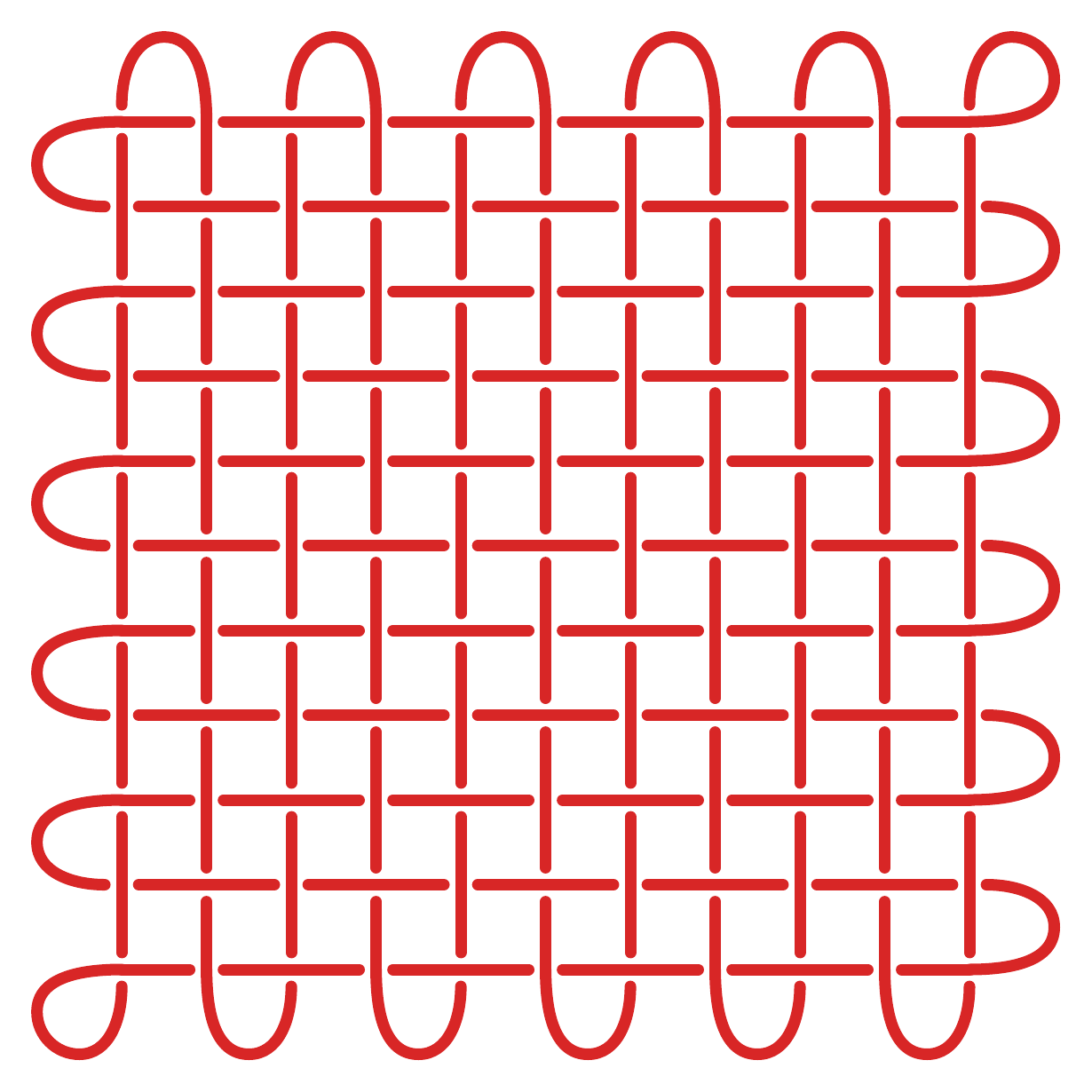}
\caption{Closing the weave off to make a knot}
\label{fig:weave}
\end{figure}

One might guess that the volume might then increase on average if you have more quadrilaterals in the complement of a diagram. However, if we fix the number of crossings, bigons, and faces of size 3, we see that the hyperbolic volume will have a tendency to \emph{decrease} when you increase the number of faces of size 4 (see Figure \ref{fig:volume_vs_quads}). The effect now though is weaker. Overall, the data then seems to support the idea that larger faces will mean greater hyperbolic volume (on average), though it is unclear why this should be the case.

\begin{figure}
\centering
\includegraphics[scale=.54]{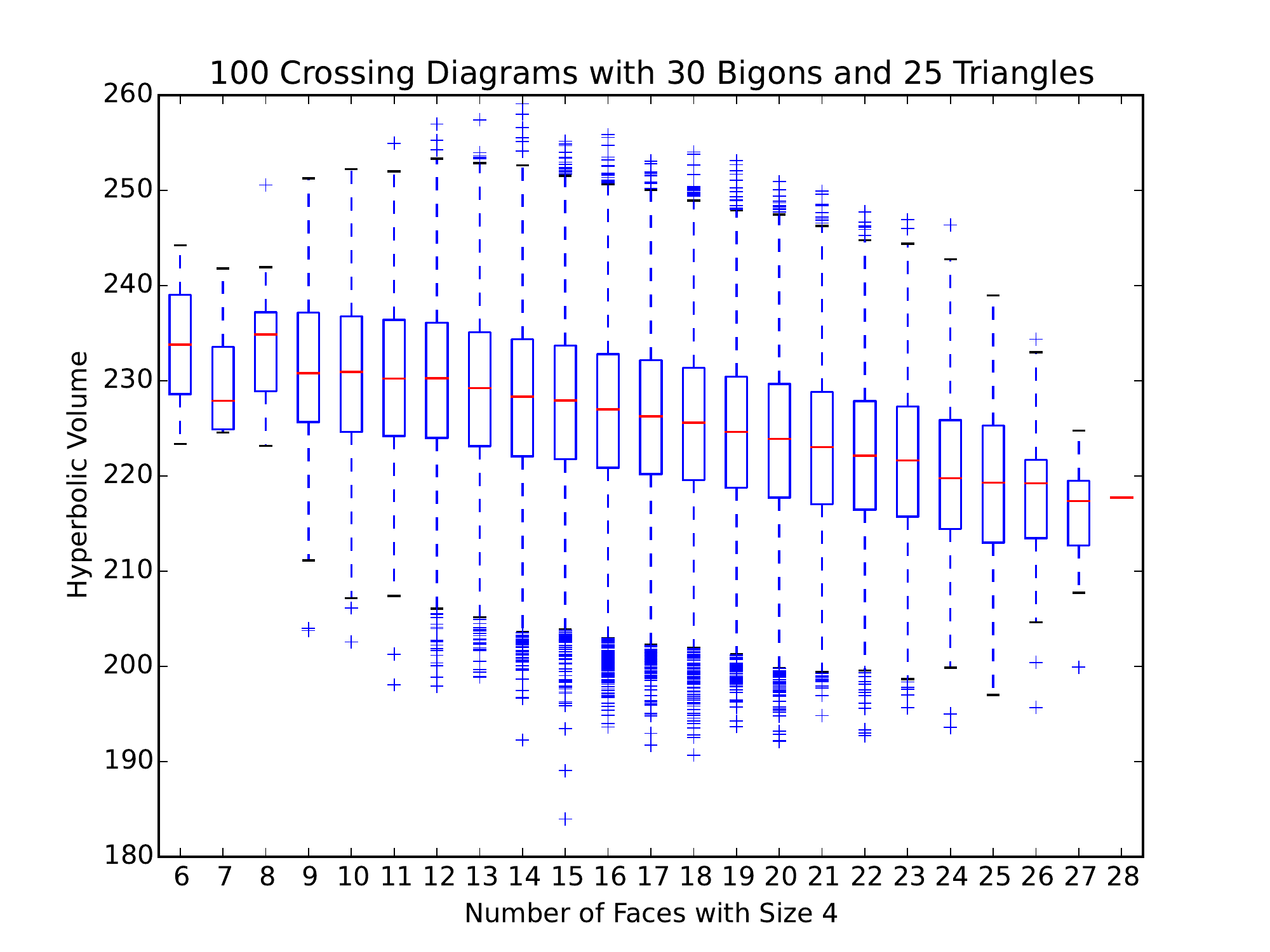}
\caption{}
\label{fig:volume_vs_quads}
\end{figure}

%\fi

\end{document}